\documentclass[12pt,reqno]{amsart}
\usepackage[margin=1in]{geometry}

\newcommand{\EquationLinkColor}{blue}
\newcommand{\CitationLinkColor}{blue}
\newcommand{\URLLinkColor}{blue}

\usepackage[T1]{fontenc}
\usepackage[utf8]{inputenc}
\usepackage{lmodern}
\usepackage{microtype}
\usepackage{xcolor}
\usepackage{booktabs}
\usepackage{float}
\usepackage{amsmath}
\usepackage{amssymb}
\usepackage{amsfonts}
\usepackage{amsthm}
\usepackage{amscd}
\usepackage{mathtools}
\usepackage{hyperref}
\hypersetup{
  colorlinks=true,
  linkcolor=\EquationLinkColor,
  citecolor=\CitationLinkColor,
  urlcolor=\URLLinkColor,
  pdftitle={All-Genus Large-Degree Asymptotics for Gromov-Witten Invariants of the Projective Plane},
  pdfauthor={Shuai Guo; Gang Tian; Tianhang Xu},
  pdfsubject={Large-degree asymptotics of Gromov-Witten invariants of the projective plane},
  pdfkeywords={Gromov-Witten invariants, projective plane, asymptotic expansions, semisimple CohFT, singularity analysis}
}

\allowdisplaybreaks
\theoremstyle{plain}
\newtheorem{theorem}{Theorem}[section]
\newtheorem{proposition}[theorem]{Proposition}
\newtheorem{lemma}[theorem]{Lemma}
\newtheorem{corollary}[theorem]{Corollary}

\theoremstyle{definition}
\newtheorem{definition}[theorem]{Definition}

\theoremstyle{remark}

\newtheorem{remark}[theorem]{Remark}

\DeclareMathOperator{\Aut}{Aut}
\DeclareMathOperator{\Cont}{Cont}
\DeclareMathOperator{\diag}{diag}
\DeclareMathOperator{\ev}{ev}
\DeclareMathOperator{\id}{id}
\DeclareMathOperator{\Li}{Li}
\DeclareMathOperator{\Tr}{Tr}
\DeclareMathOperator{\vvec}{vec}

\newcommand{\CC}{\mathbb{C}}
\newcommand{\E}{\mathcal{E}}
\newcommand{\HH}{\mathcal{H}}
\newcommand{\iu}{\mathbf{i}}
\newcommand{\N}{\mathbb{Z}_{\geq0}}
\newcommand{\one}{\mathbf{1}}
\renewcommand{\P}[1]{\mathbb{P}^{#1}}
\newcommand{\Q}{\mathbb{Q}}
\newcommand{\R}{\mathbb{R}}
\newcommand{\Z}{\mathbb{Z}}
\newcommand{\Mbar}{\overline{\mathcal M}}
\newcommand{\dd}{\mathop{}\!\mathrm{d}}
\newcommand{\abs}[1]{\left\lvert#1\right\rvert}
\newcommand{\lang}[1]{\left\langle#1\right\rangle}

\begin{document}

\title[All-Genus Asymptotics for the Projective Plane]{All-Genus Large-Degree Asymptotics for Gromov--Witten Invariants of the Projective Plane}

\author[S. Guo]{Shuai Guo}
\email{guoshuai@math.pku.edu.cn}
\address{School of Mathematical Sciences, Peking University, Beijing, 100871, China}

\author[G. Tian]{Gang Tian}
\email{gtian@math.pku.edu.cn}
\address{School of Mathematical Sciences, Peking University, Beijing, 100871, China}

\author[T. Xu]{Tianhang Xu}
\email{xth@mail.tsinghua.edu.cn}
\address{Department of Mathematical Sciences, Tsinghua University, Beijing, 100084, China}


\begin{abstract}
This is the first part of a series of papers on the large-degree asymptotics of Gromov--Witten invariants. In this paper, we prove complete large-degree asymptotic expansions, at every fixed genus, for the primary Gromov--Witten invariants of the complex projective plane. The proof uses singularity analysis to transfer the local expansions of generating functions at their dominant singularities to asymptotic expansions of Gromov--Witten invariants. The genus-zero asymptotic expansion is obtained from an analysis of the Witten--Dijkgraaf--Verlinde--Verlinde (WDVV) equation. The higher-genus cases are obtained from the Givental--Teleman reconstruction theorem for semisimple cohomological field theories and from the graph-sum formula for the $R$-matrix action. 
\end{abstract}

\maketitle

\section*{Introduction}\label{p2:sec:introduction}

The Gromov--Witten (GW) invariants arise from Gromov's compactness theory for pseudo-holomorphic curves \cite{Pseudo_holomorphic_curves_in_symplectic_manifolds} and from the mathematical formulation of quantum cohomology. Early symplectic treatments of genus-zero Gromov--Witten theory were given by Ruan--Tian \cite{A_mathematical_theory_of_quantum_cohomology} and McDuff--Salamon \cite{J-holomorphic_curves_and_quantum_cohomology}. Ruan--Tian subsequently developed the higher-genus theory in the semipositive setting \cite{Higher_genus_symplectic_invariants_and_sigma_model_coupled_with_gravity}. Several approaches to constructing Gromov--Witten invariants for general symplectic manifolds were developed by Li--Tian \cite{Virtual_moduli_cycles_and_Gromov-Witten_invariants_of_general_symplectic_manifolds}, Fukaya--Ono \cite{Arnold_conjecture_and_Gromov-Witten_invariant}, and Siebert \cite{Gromov-Witten_invariants_of_general_symplectic_manifolds}. In algebraic geometry, the virtual fundamental class is constructed by Li--Tian \cite{Virtual_moduli_cycles_and_Gromov-Witten_invariants_of_algebraic_varieties} and Behrend--Fantechi \cite{The_intrinsic_normal_cone}. Gromov--Witten theory connects intersection theory on moduli spaces, quantum cohomology, enumerative geometry, and mirror symmetry \cite{Mirror_symmetry,Mirror_symmetry_and_algebraic_geometry}.

Let $X$ be a smooth complex projective variety and $\beta\in H_2(X,\Z)$. A stable map of type $(g,n,\beta)$ is a tuple $(C,p_1,\ldots,p_n,f)$ in which $C$ is a connected nodal projective curve of arithmetic genus $g$, the $p_i$ are distinct smooth marked points, $f:C\to X$ satisfies $f_*[C]=\beta$, and the automorphism group of the tuple is finite. For projective $X$, the stable maps form a proper Deligne--Mumford stack $\Mbar_{g,n}(X,\beta)$. Its perfect obstruction theory gives a virtual fundamental class of complex dimension
\begin{equation}\label{p2:eq:general-virtual-dimension}
 (3-\dim X)(g-1)+\int_\beta c_1(T_X)+n.
\end{equation}
Define the evaluation maps $\ev_i(C,p_1,\ldots,p_n,f)=f(p_i)$. For $v_i\in H^{2q_i}(X,\Q)$, the primary Gromov--Witten invariant is defined as 
\begin{equation}\label{p2:eq:general-primary-gw}
 \lang{v_1,\ldots,v_n}_{g,n,\beta}
 :=
 \int_{[\Mbar_{g,n}(X,\beta)]^{\mathrm{vir}}}
 \prod_{i=1}^n\ev_i^*(v_i).
\end{equation}
It can be nonzero only when $\beta$ is effective and
\begin{equation}\label{p2:eq:general-dimension-constraint}
 \sum_{i=1}^nq_i
 =
 (3-\dim X)(g-1)+\int_\beta c_1(T_X)+n.
\end{equation}
The Gromov--Witten invariants above are primary, i.e., they contain no $\psi$-classes on the moduli space of stable maps.

We now specialize to $X=\P{2}$. Let $H\in H^2(\P{2},\Z)$ be the hyperplane class and let $[\ell]\in H_2(\P{2},\Z)$ be the line class. For $\beta=d[\ell]$, the virtual dimension in \eqref{p2:eq:general-virtual-dimension} is $g-1+3d+n$. For $g\geq0$ and $d\geq1$, define
\begin{equation}\label{p2:eq:invariant}
  N_{g,d}
  :=
  \lang{H^2,\ldots,H^2}_{g,3d-1+g,d[\ell]}.
\end{equation}
By the enumerative interpretation proved in \cite[Corollary~5.2]{Intersection_theory_on_M_14_and_elliptic_Gromov-Witten_invariants}, the GW invariant $N_{g,d}$ equals the number of irreducible plane curves of geometric genus $g$ and degree $d$ through $3d-1+g$ generic points. In the statement of that corollary, ``arithmetic genus'' should read ``geometric genus'': the proof uses the geometric genus, whereas an irreducible degree-$d$ plane curve has fixed arithmetic genus $(d-1)(d-2)/2$. Our purpose is to determine the asymptotic behavior of $N_{g,d}$ as $d\to\infty$ for every fixed genus $g$.

For genus-zero GW invariants $N_{0,d}$, the Witten--Dijkgraaf--Verlinde--Verlinde (WDVV) equation yields Kontsevich's recursion
\begin{equation}\label{p2:eq:kontsevich-recursion}
\begin{aligned}
 N_{0,d}
 =
 \sum_{\substack{d_1+d_2=d\\d_1,d_2\geq1}}
 N_{0,d_1}N_{0,d_2}
 \left[
 d_1^2d_2^2\binom{3d-4}{3d_1-2}
 -
 d_1^3d_2\binom{3d-4}{3d_1-1}
 \right],
 \qquad d\geq2,
\end{aligned}
\end{equation}
with $N_{0,1}=1$. Thus all $N_{0,d}$ can be recursively computed. For example,  
\[
 N_{0,1}=1,\quad N_{0,2}=1,\quad N_{0,3}=12,\quad
 N_{0,4}=620,\quad N_{0,5}=87304,\quad
 N_{0,6}=26312976,
\]
and so on. Higher genus GW invariants $N_{g,d}$ also satisfy recursive formulas, but we study their generating functions rather than their recursions directly. Our approach studies singularities of the fixed-genus generating functions and reconstructs the higher-genus generating functions from genus-zero data.

For the large-degree asymptotics of $\P{2}$, Di Francesco and Itzykson conjectured the fixed-genus leading growth of the normalized GW invariants in \cite[Footnote~2]{Quantum_intersection_rings} and claimed the genus-zero leading term in \cite[Proposition~3]{Quantum_intersection_rings}. Zinger later identified the gaps in their original genus-zero argument \cite[Section~1]{On_asymptotic_behavior_of_GW_invariants} and surveyed its resolution \cite[Sections~2.1 and~3]{Some_conjectures_on_the_asymptotic_behavior_of_Gromov-Witten_invariants}. Tian and Wei \cite[Theorems~1.1 and~1.2]{Asymptotic_of_enumerative_invariants_in_CP2} proved complete expansions in genera zero and one by analyzing the dominant singularity of the genus-zero generating function. Guzzetti \cite{The_singularity_of_Kontsevichs_solution_for_QHCP2} obtained related analytic information on the quantum-cohomological solution near the discriminant. We extend the results of Tian and Wei by proving a complete large-degree asymptotic expansion for each fixed genus $g\geq 2$.

\subsection{Main theorem}

We now state our main theorem. The first two cases are proved by Tian and Wei in \cite[Theorem~1.1]{Asymptotic_of_enumerative_invariants_in_CP2} and \cite[Theorem~1.2]{Asymptotic_of_enumerative_invariants_in_CP2}. Our present paper proves the corresponding asymptotic expansions for every fixed genus $g\geq 2$. Thus the theorem below completes the fixed-genus large-degree asymptotic picture for $\P{2}$. 

Throughout the paper, the imaginary unit is denoted by $\iu=\sqrt{-1}$.

\begin{theorem}\label{p2:thm:main}
There is a real number $x_0$ with the following properties.
\begin{enumerate}
\item There are constants $a_{0,k}$, $k\geq0$, with $a_{0,0}>0$, such that for every $N\geq1$,
\begin{equation}\label{p2:eq:g0-main}
  \frac{N_{0,d}}{(3d-1)!}
  =
  e^{-dx_0}d^{-7/2}
  \left(
    \sum_{k=0}^{N-1}a_{0,k}d^{-k}
    +
    O(d^{-N})
  \right).
\end{equation}
\item There are constants $a_{1,k}$, $k\geq1$, such that for every $N\geq1$,
\begin{equation}\label{p2:eq:g1-main}
  \frac{N_{1,d}}{(3d)!}
  =
  e^{-dx_0}d^{-1}
  \left(
    \frac1{48}
    +
    \sum_{k=1}^{N-1}a_{1,k}d^{-k+1/2}
    +
    O(d^{-N+1/2})
  \right).
\end{equation}
\item Fix $g\geq2$. There are constants $a_{g,k}$, $k\geq0$, such that for every $N\geq1$,
\begin{equation}\label{p2:eq:hg-main}
  \frac{N_{g,d}}{(3d-1+g)!}
  =
  e^{-dx_0}d^{(5g-7)/2}
  \left(
    \sum_{k=0}^{N-1}a_{g,k}d^{-k/2}
    +
    O(d^{-N/2})
  \right).
\end{equation}
Its leading coefficient is
\begin{equation}\label{p2:eq:hg-leading}
  a_{g,0}
  =
  \frac{\lang{\tau_2^{3g-3}}_g}
  {(3g-3)!\,
   (-15\iu a_{5})^{g-1}\,
   \Gamma\bigl((5g-5)/2\bigr)}
  >0,
\end{equation}
where
\begin{equation}\label{p2:eq:psi-integral}
  \lang{\tau_2^{3g-3}}_g
  :=
  \int_{\Mbar_{g,3g-3}}
  \psi_1^2\cdots\psi_{3g-3}^2
\end{equation}
and $a_5\in\iu\R_{>0}$ is the constant defined in \eqref{p2:eq:a5}; here $\Gamma$ denotes Euler's gamma function. Moreover,
\begin{equation}\label{p2:eq:hg-vanishing}
 a_{g,k}=0
 \quad\text{whenever}\quad
 \frac{5g-7-k}{2}\in\Z_{\leq-1}.
\end{equation}
\end{enumerate}
Consequently,
\begin{equation}\label{p2:eq:root-limit}
  \lim_{d\to\infty}
  \sqrt[d]{\frac{N_{g,d}}{(3d-1+g)!}}
  =
  e^{-x_0}
\end{equation}
for every fixed $g\geq0$.
\end{theorem}

The constant $x_0$ is the common abscissa of convergence of the series
\begin{equation}\label{p2:eq:Fg-definition}
  F_g(z) := \sum_{d\geq1}
  \frac{N_{g,d}}{(3d-1+g)!}e^{dz}.
\end{equation}
Numerically, $x_0\approx 1.9804$ and $a_{5}\approx \iu\cdot 5.7057$.

Our proof has three steps. First, the WDVV equation reduces to an analytic differential equation satisfied by the genus-zero generating function. Its solution reaches a singularity at a finite real value $x_0$, near which it has a locally convergent square-root expansion. Second, the Givental--Teleman reconstruction theorem for semisimple homogeneous cohomological field theories \cite{Gromov-Witten_invariants_and_quantization_of_quadratic_hamiltonians,The_structure_of_2D_semi-simple_field_theories} expresses the higher-genus potentials as stable-graph sums via the $R$-matrix action, which can be used to compute the Laurent--Puiseux expansions at the dominant singularity $x_0$ of higher-genus generating functions. Third, local singularity analysis transfers the Laurent--Puiseux expansions near $x_0$ to asymptotics of the GW invariants.

Parts of this strategy may extend to other Fano varieties, including $\P{3}$, and del Pezzo surfaces such as $\P{1}\times\P{1}$ and the blow-ups of $\P{2}$ at $1\le n \le 8$ points in general position. These cases will be investigated in subsequent papers in this series.

\subsection{Plan of the paper}

This paper is organized as follows. Section~\ref{p2:sec:framework} introduces the necessary background and tools, including the Givental--Teleman reconstruction and the transfer theorem in singularity analysis. Section~\ref{p2:sec:g0} establishes the genus-zero result. Sections~\ref{p2:sec:g1} and~\ref{p2:sec:hg} reconstruct the remaining genera and determine explicitly the leading terms of their asymptotic expansions.

\section{Reconstruction and coefficient transfer}
\label{p2:sec:framework}

This section introduces the tools used in our asymptotic analysis of GW invariants in every genus. The genus-zero part is governed by the Frobenius manifold and the WDVV equations. The Gromov--Witten classes form a cohomological field theory (CohFT) in the sense of \cite{Gromov-Witten_classes_quantum_cohomology_and_enumerative_geometry}. On the semisimple locus, the higher-genus theory is recovered from its genus-zero data by the Givental--Teleman reconstruction theorem for homogeneous semisimple CohFTs \cite{Gromov-Witten_invariants_and_quantization_of_quadratic_hamiltonians,The_structure_of_2D_semi-simple_field_theories}, and the resulting $R$-matrix action is expressed as a stable-graph sum \cite[Definitions~2.2 and~2.13]{Relations_on_Mgn_via_3-spin_structures}. Finally, we record the coefficient-transfer theorem from singularity analysis.

\subsection{Frobenius manifolds and the WDVV equations}

\begin{definition}[{\cite[Definition~1.2]{Geometry_of_2D_topological_field_theories}}]\label{p2:def:frobenius}
A complex analytic manifold $M$ is a Frobenius manifold if it is equipped with the following data.
\begin{enumerate}
\item It carries a flat, nondegenerate, symmetric bilinear form $\eta$ on the tangent bundle, called the Frobenius metric.
\item Every tangent space is equipped with a commutative, associative multiplication $*$, called the Frobenius multiplication, which depends analytically on the point, has unit $e$, and satisfies
\[
 \eta(u*v,w)=\eta(u,v*w).
\]
If the resulting algebra is semisimple on a dense open set, then $M$ is called semisimple.
\item The unit $e$ is flat for the Levi--Civita connection: $\nabla e=0$.
\item If $c(u,v,w)=\eta(u*v,w)$, then $\nabla_zc(u,v,w)$ is symmetric in $u,v,w,z$.
\item There are an Euler vector field $E$ and a number $\delta$, called the conformal dimension, such that
\[
 \nabla\nabla E=0,\qquad
 \mathcal L_E\eta=(2-\delta)\eta,\qquad
 \mathcal L_E(*)=*.
\]
\end{enumerate}
\end{definition}

In this paper, we restrict attention to Frobenius manifolds for which $\nabla E$ is diagonalizable. Choose local flat coordinates $t_0,\ldots,t_r$ for which $e=\partial_{t_0}$. They may be chosen so that
\begin{equation}\label{p2:eq:general-euler-field}
 E
 =
 \sum_{\alpha=0}^{r}
 \bigl((1-q_\alpha)t_\alpha+\rho_\alpha\bigr)
 \partial_{t_\alpha},
 \qquad
 q_0=0,
\end{equation}
where $\rho_\alpha\neq0$ only when $q_\alpha=1$. Locally, the Frobenius structure is determined by a prepotential $F(t_0,\ldots,t_r)$. Denote
\[
 F_{\alpha_1\cdots\alpha_k}
 :=
 \partial_{t_{\alpha_1}}\cdots
 \partial_{t_{\alpha_k}}F,
 \qquad
 (\eta^{\alpha\beta}):=(\eta_{\alpha\beta})^{-1}.
\]
The Witten--Dijkgraaf--Verlinde--Verlinde (WDVV) equations are third-order partial differential equations satisfied by the prepotential $F(t_0,\ldots,t_r)$:
\begin{equation}\label{p2:eq:wdvv}
  \sum_{\lambda,\nu}
  F_{\alpha\beta\lambda}\eta^{\lambda\nu}
  F_{\nu\gamma\delta}
  =
  \sum_{\lambda,\nu}
  F_{\alpha\gamma\lambda}\eta^{\lambda\nu}
  F_{\nu\beta\delta},
\end{equation}
where $\eta_{\alpha\beta}=F_{0\alpha\beta}$ is constant and nondegenerate. The WDVV equations are equivalent to the associativity of Frobenius multiplication 
\begin{equation}\label{p2:eq:frobenius-product}
  \partial_{t_\alpha}*\partial_{t_\beta}
  =
  \sum_{\gamma,\nu}
  F_{\alpha\beta\gamma}\eta^{\gamma\nu}
  \partial_{t_\nu}.
\end{equation}
The Euler vector field gives the quasi-homogeneity relation
\begin{equation}\label{p2:eq:prepotential-homogeneity}
 E(F)=(3-\delta)F+\text{a polynomial of degree at most two}.
\end{equation}
Conversely, the prepotential $F$ satisfying \eqref{p2:eq:wdvv}, the constant metric $\eta_{\alpha\beta}=F_{0\alpha\beta}$, the unit $e=\partial_{t_0}$, and \eqref{p2:eq:prepotential-homogeneity} determine the local Frobenius structure.

For a smooth projective variety $X$ with $H^{\mathrm{odd}}(X,\CC)=0$, take $\HH=H^*(X,\CC)$ with the Poincar\'e pairing $\eta(u,v) := \int_X u\cup v$. Choose a homogeneous basis
\[
 T_0=\one,T_1,\ldots,T_r,
 \qquad
 T_i\in H^{2q_i}(X,\CC),
\]
where $\one$ denotes the multiplicative identity. Suppose $T_1,\ldots,T_m$ are divisor classes. Let $Q^\beta$ denote the Novikov monomial of $\beta$, and understand every curve-class sum in the Novikov completion. If $t=\sum_it_iT_i$, then the genus-$g$ Gromov--Witten potential is
\begin{equation}\label{p2:eq:general-gw-potential}
 F_g^X(t)
 :=
 \sum_{n\geq0}\sum_\beta
 \frac1{n!}\lang{t,\ldots,t}_{g,n,\beta}Q^\beta.
\end{equation}
The degree-zero, fundamental-class, and divisor axioms give
\begin{equation}\label{p2:eq:genus-zero-divisor-form}
\begin{aligned}
 F_0^X(t)
 ={}&
 \frac16\int_Xt^3 +
 \sum_{\beta\neq0}
 Q^\beta
 \exp\left(\int_\beta\sum_{i=1}^{m}t_iT_i\right)
 \sum_{n\geq0}\frac1{n!}
 \left\langle
 \sum_{i=m+1}^{r}t_iT_i,\ldots,
 \sum_{i=m+1}^{r}t_iT_i
 \right\rangle_{0,n,\beta}.
\end{aligned}
\end{equation}
The WDVV equation for $F_0^X$ is precisely associativity of the big quantum product. Its unit is $\one$, its flat metric is the Poincar\'e pairing, and its Euler vector field is
\begin{equation}\label{p2:eq:projective-euler-field}
 E
 =
 \sum_{i=0}^{r}(1-q_i)t_i\partial_{t_i}
 +
 c_1(T_X),
\end{equation}
where a cohomology class is identified with the corresponding constant flat vector field. Thus the quantum cohomology of $X$, as a formal germ or on an analytic convergence domain, is a Frobenius manifold of conformal dimension $\dim X$.

Let $\E=E*$ denote multiplication by the Euler vector field, both as an endomorphism and as its matrix in the flat basis $\{\partial_{t_0},\ldots,\partial_{t_r}\}$. Let
\begin{equation}\label{p2:eq:hodge-grading}
 \mu
 :=
 \left(1-\frac{\delta}{2}\right)\id-\nabla E
\end{equation}
be the Hodge grading operator. In the flat basis, $\mu=\diag(q_0-\delta/2,\ldots,q_r-\delta/2)$. Differentiating \eqref{p2:eq:prepotential-homogeneity} gives the $(\alpha,\beta)$-entry of the matrix $\E$ in the flat basis: 
\begin{equation}\label{p2:eq:euler-entries}
 \E^\alpha_{ \beta}
 =
 \sum_\gamma
 \eta^{\alpha\gamma}
 (q_\beta+q_\gamma+1-\delta)F_{\beta\gamma}
 +
 \rho^\alpha_{ \beta},
\end{equation}
where $F_{\beta\gamma}$ is a second derivative of $F$ and $\rho^\alpha_{ \beta}$ is the constant term produced by quasi-homogeneity. The WDVV equations give
\begin{equation}\label{p2:eq:flat-euler-derivative}
 \partial_{t_\alpha}\E
 =
 (\partial_{t_\alpha}*)+
 [\partial_{t_\alpha}*,\mu].
\end{equation}
Therefore, 
\begin{equation}\label{p2:eq:euler-wdvv}
  E^{*k}(\E)=\E^k+[\E^k,\mu],
  \qquad k\geq0.
\end{equation}
Here $E^{*k}(\E)$ denotes the directional derivative of the matrix-valued function $\E$ along the Frobenius power $E^{*k}$, whereas $\E^k$ denotes the ordinary matrix power.

Define the cyclic transition matrix by
\begin{equation}\label{p2:eq:Y-definition}
  (\one,E,E^{*2},\ldots,E^{*r})
  =
  (\partial_{t_0},\ldots,\partial_{t_r})Y,
  \qquad
  \mathcal Y:=\det Y.
\end{equation}
The $(\alpha,k)$-entry of $Y$ is $(\E^k)^\alpha_{ 0}$. On $\mathcal Y\neq0$, the Euler powers form a basis and
\begin{equation}\label{p2:eq:euler-wdvv-flat-system}
 \partial_{t_\gamma}\E^\alpha_{ \beta}
 =
 \sum_{k=0}^{r}
 \bigl((Y^{\mathsf T})^{-1}\bigr)_{\gamma k}
 \bigl(\E^k+[\E^k,\mu]\bigr)^\alpha_{ \beta}.
\end{equation}
Thus the WDVV equation and homogeneity give a closed first-order rational system for the flat-basis entries of $\E$, with $\mathcal Y$ as its only explicit denominator.

\subsection{Cohomological field theories}

A Frobenius manifold records only genus-zero data. To organize the higher-genus part, we recall the notion of a cohomological field theory (CohFT).

\begin{definition}\label{p2:def:cohft}
Let $\HH$ be a finite-dimensional complex vector space with a nondegenerate symmetric bilinear form $\eta$ and a distinguished vector $\one$. A cohomological field theory with unit is a collection
\[
 \Omega_{g,n}
 \in
 H^*(\Mbar_{g,n})\otimes(\HH^*)^{\otimes n},
 \qquad
 2g-2+n>0,
\]
satisfying the following axioms.
\begin{enumerate}
\item $\Omega_{g,n}$ is invariant under the simultaneous action of $S_n$ on the marked points and tensor factors.
\item For the gluing maps
\[
 q:\Mbar_{g-1,n+2}\longrightarrow\Mbar_{g,n},
\]
\[
 r:\Mbar_{g_1,n_1+1}\times\Mbar_{g_2,n_2+1}
 \longrightarrow\Mbar_{g,n},
\]
the pullbacks of $\Omega$ are obtained by contracting the two new factors with $\eta^{-1}=\sum_\alpha e_\alpha\otimes e^\alpha$, where $\{e_\alpha\}$ is any basis and $\{e^\alpha\}$ is its $\eta$-dual. Explicitly, for $v_1,\ldots,v_n\in\HH$,
\[
 q^*\Omega_{g,n}(v_1,\ldots,v_n)
 =
 \sum_\alpha
 \Omega_{g-1,n+2}(v_1,\ldots,v_n,e_\alpha,e^\alpha).
\]
For a stable splitting $I\sqcup J=\{1,\ldots,n\}$, $g_1+g_2=g$, $n_1=|I|$, and $n_2=|J|$, one has
\[
\begin{aligned}
 r^*\Omega_{g,n}(v_1,\ldots,v_n)
 ={}&
 \sum_\alpha
 \Omega_{g_1,n_1+1}(v_I,e_\alpha) \otimes
 \Omega_{g_2,n_2+1}(v_J,e^\alpha),
\end{aligned}
\]
where $v_I$ and $v_J$ are written in the induced order.
\item If $p:\Mbar_{g,n+1}\to\Mbar_{g,n}$ forgets the last marking, then
\[
 \Omega_{g,n+1}(v_1,\ldots,v_n,\one)
 =
 p^*\Omega_{g,n}(v_1,\ldots,v_n),
\]
and
\[
 \Omega_{0,3}(v_1,v_2,\one)=\eta(v_1,v_2).
\]
\end{enumerate}
\end{definition}

The three-point tensor defines a Frobenius algebra by
\begin{equation}\label{p2:eq:cohft-product}
 \eta(v_1*v_2,v_3)=\Omega_{0,3}(v_1,v_2,v_3).
\end{equation}
For $t\in\HH$, the shifted CohFT is
\begin{equation}\label{p2:eq:shifted-cohft}
 \Omega^t_{g,n}(v_1,\ldots,v_n)
 :=
 \sum_{m\geq0}\frac1{m!}
 (p_m)_*
 \Omega_{g,n+m}(v_1,\ldots,v_n,t,\ldots,t),
\end{equation}
where $p_m$ forgets the last $m$ markings and stabilizes. The sum is understood in the $t$-adic and Novikov completion, or analytically on a domain where the shifted potential converges. The three-point part of the shifted family $\Omega^t$ is the Frobenius-manifold product. $\Omega^t$ is semisimple if the Frobenius algebra is semisimple. The CohFT is homogeneous if it is compatible with an Euler vector field in the sense of \cite{The_structure_of_2D_semi-simple_field_theories}.

A topological field theory (TFT) is a CohFT taking values in $H^0(\Mbar_{g,n})$. Its values are determined by the Frobenius algebra. If $\{e_\alpha\}$ is any basis and $\{e^\alpha\}$ is its 
$\eta$-dual, define the handle element
\[
 \Delta:=\sum_\alpha e_\alpha*e^\alpha.
\]
Then
\begin{equation}\label{p2:eq:TFT}
  \omega_{g,n}(v_1,\ldots,v_n)
  =
  \eta(v_1*\cdots*v_n,\Delta^{*g}).
\end{equation}
Indeed, consider the pull-back to a maximally nodal stable curve. Repeated separating gluings reduce every genus-zero component to $\eta(v_1*\cdots*v_m,\one)$, and each nonseparating node inserts $\sum_\alpha e_\alpha\otimes e^\alpha$, hence multiplies by $\Delta$.

For a smooth projective variety $X$, let $\pi:\Mbar_{g,n}(X,\beta)\to\Mbar_{g,n}$ be the stabilization and let $\ev=(\ev_1,\ldots,\ev_n)$ be the product of evaluation maps. Write $\operatorname{PD}$ for Poincar\'e duality. The Gromov--Witten class is
\begin{equation}\label{p2:eq:gw-cohft-class}
\begin{aligned}
 I_{g,n,\beta}(v_1,\ldots,v_n)
 &:=
 \operatorname{PD}\!\left[
 \pi_*\left(
 \ev^*(v_1\otimes\cdots\otimes v_n)
 \cap[\Mbar_{g,n}(X,\beta)]^{\mathrm{vir}}
 \right)
 \right],\\
 \Omega_{g,n}
 &:=
 \sum_\beta Q^\beta I_{g,n,\beta}.
\end{aligned}
\end{equation}
These classes form a CohFT\@. Its shifted classes recover the primary potentials:
\begin{align}
 F_g^X(t)
 &=
 \int_{\Mbar_{g,0}}\Omega^t_{g,0},
 \qquad g\geq2,
 \label{p2:eq:potential-as-shifted-cohft}\\
 \partial_{t_\alpha}F_1^X(t)
 &=
 \int_{\Mbar_{1,1}}
 \Omega^t_{1,1}(\partial_{t_\alpha}).
 \label{p2:eq:F1-as-shifted-cohft}
\end{align}

For the Gromov--Witten CohFT, the dimension and divisor axioms give homogeneity with respect to the Euler vector field \eqref{p2:eq:projective-euler-field}. For each fixed $t$, put
\[
 \omega_{g,n}
 :=
 \bigl[\Omega^t_{g,n}\bigr]_{H^0(\Mbar_{g,n})}
\]
for the topological part of the shifted theory. We suppress its dependence on $t$: this is the TFT associated with the Frobenius algebra at $t$, not the $t$-shift of a fixed TFT\@. Teleman's theorem \cite{The_structure_of_2D_semi-simple_field_theories} gives, at every semisimple point $t$,
\begin{equation}\label{p2:eq:teleman-applied}
 \Omega^t=R_t\mathbin{.}\omega,
\end{equation}
where $R_t$ is the unique $R$-matrix compatible with homogeneity, characterized by the recursion \eqref{p2:eq:R-recursion} below. The notation $R_t\mathbin{.}\omega$ denotes the unit-preserving Givental action, including the translation defined in \eqref{p2:eq:translation}. Next, we recall the $R$-matrix action on a CohFT in the form of a stable-graph sum.

\subsection{The \texorpdfstring{$R$}{R}-matrix graph sum}

We use Givental's $R$-matrix action
\cite{Gromov-Witten_invariants_and_quantization_of_quadratic_hamiltonians} in its unit-preserving CohFT form \cite[Definition~2.13]{Relations_on_Mgn_via_3-spin_structures}. An $R$-matrix is a series
\[
 R(\xi)=R_0+R_1\xi+R_2\xi^2+\cdots,
 \qquad R_0=\id,
\]
in $\operatorname{End}(\HH)[[\xi]]$, satisfying
\begin{equation}\label{p2:eq:R-symplectic}
 R(\xi)R^*(-\xi)=\id.
\end{equation}
The adjoint is defined by $\eta(Av,w)=\eta(v,A^*w)$. Consequently, 
\begin{equation}\label{p2:eq:R-inverse}
 R^{-1}(\xi)
 =
 R^*(-\xi)
 =
 \sum_{m\geq0}(-1)^mR_m^*\xi^m
\end{equation}
and
\begin{equation}\label{p2:eq:R-convolution}
 \sum_{a+b=m}(-1)^bR_aR_b^*
 =
 \sum_{a+b=m}(-1)^aR_a^*R_b
 =
 \delta_{m0}\id.
\end{equation}
In particular,
\[
 R_1^*=R_1,\qquad
 R_2^*=R_1^2-R_2.
\]
A superscript $*$ on an endomorphism denotes this adjoint; expressions such as $E^{*k}$ and $\Delta^{*m}$ denote Frobenius powers.

The unit-preserving translation is
\begin{equation}\label{p2:eq:translation}
  \mathsf T(\psi)
  :=
  \psi\bigl(\one-R^{-1}(\psi)\one\bigr)
  =
  \sum_{m\geq2}\mathsf T_m\psi^m,
  \qquad
  \mathsf T_m:=(-1)^mR_{m-1}^*\one.
\end{equation}
Thus $\mathsf T_2=R_1\one$ and $\mathsf T_3=(R_2-R_1^2)\one$. The translation acts on a CohFT by
\begin{equation}\label{p2:eq:translated-cohft}
\begin{aligned}
 (\mathsf T\Omega)_{g,n}(v_1,\ldots,v_n)
 ={}&
 \sum_{m\geq0}
 \sum_{i_1,\ldots,i_m\geq2}
 \frac1{m!}(p_m)_*\\
 &\left(
 \prod_{j=1}^{m}\psi_{n+j}^{i_j}\,
 \Omega_{g,n+m}
  (v_1,\ldots,v_n,\mathsf T_{i_1},\ldots,\mathsf T_{i_m})
 \right).
\end{aligned}
\end{equation}
Cohomological degree makes this sum finite.

A stable graph $\Gamma$ of type $(g,n)$ consists of a finite connected graph with $n$ labeled legs, together with an assignment $g_v\in\mathbb Z_{\geq0}$ to every vertex $v$, such that
\[
 2g_v-2+n_v>0,
 \qquad
 g=h^1(\Gamma)+\sum_v g_v.
\]
Here $n_v$ counts both legs and half-edges incident to $v$. Write
\[
 \Mbar_\Gamma=\prod_v\Mbar_{g_v,n_v},
 \qquad
 \operatorname{gl}_\Gamma:\Mbar_\Gamma\longrightarrow\Mbar_{g,n}
\]
for the gluing morphism. The graph contribution $\Cont_\Gamma(v_1,\ldots,v_n)\in H^*(\Mbar_{g,n})$ is constructed as follows:
\begin{enumerate}
\item put $(\mathsf T\Omega)_{g_v,n_v}$ at every vertex;
\item put $R^{-1}(\psi_l)v_l$ at the $l$-th leg;
\item at an edge with half-edge cotangent classes $\psi',\psi''$, put
\begin{equation}\label{p2:eq:edge}
  \frac{
    \eta^{-1}
    -R^{-1}(\psi')\eta^{-1}R^{-1}(\psi'')^{\mathsf T}
  }{\psi'+\psi''};
\end{equation}
\item contract all factors and push the resulting class forward by $\operatorname{gl}_\Gamma$.
\end{enumerate}
Here ${}^{\mathsf T}$ denotes ordinary matrix transpose in the fixed flat basis. An automorphism of $\Gamma$ preserves the vertex genera, the edges, and every labeled leg. The numerator in \eqref{p2:eq:edge} vanishes at $\psi''=-\psi'$ by \eqref{p2:eq:R-symplectic}, so the quotient is a formal power series. In a basis $\{e_k\}$ with dual basis $\{e^k\}$, it is
\begin{equation}\label{p2:eq:edge-expanded}
 \sum_{\alpha,\beta\geq0}(\psi')^\alpha(\psi'')^\beta
 \sum_{p=0}^{\beta}\sum_k
 (-1)^{p+\alpha}
 R^*_{\alpha+\beta+1-p}e_k\otimes R_p^*e^k.
\end{equation}
Let $G_{g,n}$ denote the set of isomorphism classes of stable graphs of type $(g,n)$. The $R$-matrix action is the graph sum
\begin{equation}\label{p2:eq:R-action-graph-sum}
 (R\mathbin{.}\Omega)_{g,n}(v_1,\ldots,v_n)
 :=
 \sum_{\Gamma\in G_{g,n}}
 \frac1{\abs{\Aut\Gamma}}
 \Cont_\Gamma(v_1,\ldots,v_n).
\end{equation}

Teleman's classification theorem \cite[Theorem~1, Proposition~8.12, and Theorem~8.15]{The_structure_of_2D_semi-simple_field_theories} says that on the semisimple locus a homogeneous CohFT is reconstructed from its topological part by the unique $R$-matrix compatible with homogeneity, with its action on the CohFT understood in the unit-preserving sense of \eqref{p2:eq:translation} and \eqref{p2:eq:R-action-graph-sum}. The $R$-matrix satisfies
\begin{equation}\label{p2:eq:R-recursion}
  [R_{n+1},\E]=(n\cdot\id+\mu)R_n,
  \qquad n\geq0.
\end{equation}
This reconstructs higher-genus CohFT classes from the genus-zero Frobenius-manifold data.

Applying \eqref{p2:eq:R-action-graph-sum} to the shifted topological part and using \eqref{p2:eq:potential-as-shifted-cohft}--\eqref{p2:eq:F1-as-shifted-cohft} gives
\begin{align}
 \partial_{t_\alpha}F_1^X(t)
 &=
 \sum_{\Gamma\in G_{1,1}}
 \frac1{\abs{\Aut\Gamma}}
 \int_{\Mbar_{1,1}}\Cont_\Gamma(\partial_{t_\alpha}),
 \label{p2:eq:graph-F1}\\
 F_g^X(t)
 &=
 \sum_{\Gamma\in G_{g,0}}
 \frac1{\abs{\Aut\Gamma}}
 \int_{\Mbar_{g,0}}\Cont_\Gamma,
 \qquad g\geq2.
 \label{p2:eq:graph-Fg}
\end{align}

\subsection{The polynomial structure and the discriminant}

We next rewrite the abstract graph sum as rational expressions in flat coordinates. This yields the polynomial structure of the higher-genus potentials.

Let $\mathcal D$ be the discriminant of the characteristic polynomial of $\E=E*$. Near a semisimple point there are canonical coordinates $u_0,\ldots,u_r$ for which
\[
 \partial_{u_i}*\partial_{u_j}
 =
 \delta_{ij}\partial_{u_i},
 \qquad
 E=\sum_i u_i\partial_{u_i},
 \qquad
 \one=\sum_i\partial_{u_i}.
\]
The canonical basis is orthogonal for $\eta$, and the $u_i$ are the eigenvalues of $\E$.  We call a semisimple point \emph{Euler-tame} if the $u_i$ are pairwise distinct, equivalently $\mathcal D\neq0$. This is stronger than semisimplicity of the Frobenius algebra.

\begin{proposition}\label{p2:prop:reconstruction}
For a homogeneous semisimple CohFT, the following quantities are, on the locus $\mathcal Y\mathcal D\neq0$, rational expressions in the flat-basis entries of $\E$: the flat-basis entries of $R_n$, the first-order derivatives of the genus-one primary potential, and the primary potentials $F_g$ for $g\geq2$. The denominator of each such expression is a product of powers of $\mathcal Y$ and $\mathcal D$.
\end{proposition}

\begin{proof}
The transition matrix \eqref{p2:eq:Y-definition} gives 
\begin{equation}\label{p2:eq:flat-from-euler}
 \partial_{t_\alpha}
 =
 \sum_{k=0}^{r}
 ((Y^{\mathsf T})^{-1})_{\alpha k}E^{*k}.
\end{equation}
Hence
\[
 (\partial_{t_\alpha}*)
 =
 \sum_{k=0}^{r}
 ((Y^{\mathsf T})^{-1})_{\alpha k}\E^k,
\]
so every flat multiplication matrix is rational in flat-basis entries of $\E$, with only powers of $\mathcal Y$ in the denominator.

Write $\widetilde R_n$ and $\widetilde\mu$ for the matrices of $R_n$ and $\mu$ in the canonical basis $\{\partial_{u_0},\ldots,\partial_{u_r}\}$. In the canonical basis, the recursion
\eqref{p2:eq:R-recursion} reads
\[
 [\widetilde R_{n+1},U]
  =
  (n\cdot\id+\widetilde\mu)\widetilde R_n,
 \qquad
 U=\diag(u_0,\ldots,u_r).
\]
Since $\widetilde\mu$ is skew-symmetric with respect to $\eta$, its diagonal entries vanish. Then the off-diagonal entries of $\widetilde R_{n+1}$ are determined. The diagonal part of the next recursion determines its diagonal entries. Thus the homogeneous recursion uniquely determines all $R_n$ from $R_0=\id$.

Denote by $\vvec$ the operator that stacks the columns of a matrix. Denote by $I$ the identity matrix of size $r+1$. Define
\begin{align}
 \mathcal A_n
 &:=
 \E^{\mathsf T}\otimes I-I\otimes\E  +
 \sum_{k=0}^{r}
 \vvec(\E^k)
 \vvec\!\left(
  ((n+1)\cdot\id+\mu)^{\mathsf T}(\E^k)^{\mathsf T}
 \right)^{\mathsf T}.
 \label{p2:eq:An}
\end{align}
The recursion of the $R$-matrix can be written as
\begin{equation}\label{p2:eq:An-system}
 \mathcal A_n\vvec(R_{n+1})
 =
 \vvec((n\cdot\id+\mu)R_n).
\end{equation}
This encodes
\[
 \Tr\bigl(\E^k((n+1)\cdot\id+\mu)R_{n+1}\bigr)=0,
 \qquad 0\leq k\leq r.
\]
In the canonical basis and then the elementary-matrix basis,
$\mathcal A_n$ is block lower triangular.  Its off-diagonal block has
diagonal entries $u_j-u_i$, while its diagonal block is
$(n+1)VV^{\mathsf T}$, where
$V=(u_i^k)_{0\leq i,k\leq r}$ is the Vandermonde matrix.  Therefore, 
\begin{equation}\label{p2:eq:R-system-determinant}
  \det\mathcal A_n
  =
  (-1)^{r(r+1)/2}(n+1)^{r+1}\mathcal D^2.
\end{equation}
Note that $\E^*=\E$, $\mu^*=-\mu$. Induction in \eqref{p2:eq:An-system} implies that every $R_n$ and its adjoint $R_n^*$ are in the localization by $\mathcal D$.

The TFT tensors are polynomial in the flat multiplication matrices. Then by \eqref{p2:eq:graph-F1}--\eqref{p2:eq:graph-Fg} and
\eqref{p2:eq:flat-from-euler}, the proposition is proved.
\end{proof}

\begin{remark}
  There is a relation between the two denominators $\mathcal{Y}$ and $\mathcal{D}$. Define
  \[
    \mathcal U:=\det(\Delta*).
  \]
  At a semisimple point, write $\eta(\partial_{u_i},\partial_{u_j})=\delta_{ij}\eta_{ii}$. Then
  \[
    \Delta=\sum_i\eta_{ii}^{-1}\partial_{u_i},
    \qquad
    \mathcal U=\prod_i\eta_{ii}^{-1}.
  \]
  If $C$ is the canonical-to-flat change-of-basis matrix, then $Y=CV$, while
  \[
    C^{\mathsf T}(\eta_{\alpha\beta})C
    =
    \diag(\eta_{00},\ldots,\eta_{rr}).
  \]
  It follows that
  \begin{equation}\label{p2:eq:Y-D-U-relation}
    \mathcal D
    =
    \det(\eta_{\alpha\beta})\,\mathcal U\,\mathcal Y^2.
  \end{equation}
\end{remark}

The next proposition implies that $\mathcal U$ detects semisimplicity, while $\mathcal D=0$ only detects a collision of eigenvalues of $E*$.
\begin{proposition}\label{p2:prop:handle-semisimple}
Let $A$ be a finite-dimensional commutative Frobenius algebra over $\CC$. Then $A$ is semisimple if and only if multiplication by the handle element $\Delta$ is invertible.
\end{proposition}

\begin{proof}
If $m_x$ denotes multiplication by $x$, a computation of the trace in dual bases gives
\[
 \Tr(m_{x*y})
 =
 \eta(x*y,\Delta)
 =
 \eta(x,m_\Delta y).
\]
Since $\eta$ is nondegenerate, the trace pairing $(x,y)\mapsto\Tr(m_{x*y})$ is nondegenerate exactly when $m_\Delta$ is invertible. A finite-dimensional commutative algebra over $\CC$ is semisimple exactly when its trace pairing is nondegenerate.
\end{proof}

\begin{proposition}\label{p2:prop:D-propagation}
For each flat coordinate $t_\alpha$, there is a polynomial $p_\alpha(\E)$ in the flat-basis entries of $\E$ such that
\begin{equation}\label{p2:eq:D-propagation}
  \mathcal Y\,\partial_{t_\alpha}\mathcal D
  =
  p_\alpha(\E)\mathcal D.
\end{equation}
\end{proposition}

\begin{proof}
Equation \eqref{p2:eq:Y-definition} is equivalent to
\[
 Y^{\mathsf T}
 \begin{pmatrix}
 \partial_{t_0}\\ \vdots\\ \partial_{t_r}
 \end{pmatrix}
 =
 \begin{pmatrix}
  \one\\ E\\ \vdots\\ E^{*r}
 \end{pmatrix}.
\]
Multiplication by $\operatorname{adj}(Y^{\mathsf T})$ expresses every $\mathcal Y\partial_{t_\alpha}$ as a polynomial linear combination of $\one,E,\ldots,E^{*r}$. In canonical coordinates we have $E^{*k}= \sum_{i=0}^{r} u_i^k \partial_{u_i}$ and $\mathcal{D}=\prod_{i<j} (u_i-u_j)^2$. On the Euler-tame locus, by straightforward computation,
\begin{equation}\label{p2:eq:D-euler-power}
 E^{*k}(\log\mathcal D)
 =
 2\sum_{i<j}\frac{u_i^k-u_j^k}{u_i-u_j}.
\end{equation}
The right-hand side is a symmetric polynomial in the $u_i$, hence a polynomial in the coefficients of $\det(\lambda\id-\E)$, and therefore in the entries of $\E$. This proves \eqref{p2:eq:D-propagation} on $\mathcal D\neq0$; analyticity extends the identity across $\mathcal D=0$.
\end{proof}

\begin{remark}
Let $t:[s_0,s_1]\to\{\mathcal Y\neq0\}$ be a piecewise $C^1$ path. More generally, one may allow $\mathcal Y(t(s_1))=0$, provided that $\mathcal Y(t(s))\neq0$ for $s_0\leq s<s_1$ and that
\[
 a(s):=
 \sum_{\alpha=0}^{r}
 \frac{p_\alpha(\E(t(s)))}{\mathcal Y(t(s))}
 \dot t_\alpha(s)
\]
belongs to $L^1([s_0,s_1])$. Equation \eqref{p2:eq:D-propagation} and the chain rule give
\[
 \frac{\dd}{\dd s}\mathcal D(t(s))
 =a(s)\mathcal D(t(s))
\]
for $s_0\leq s<s_1$. Hence, by integration and continuity at the endpoint,
\begin{equation}\label{p2:eq:D-path}
\begin{aligned}
 \mathcal D(t(s_1))
 ={}&
 \mathcal D(t(s_0))
 \exp\left(
 \int_{s_0}^{s_1}
 \sum_{\alpha=0}^{r}
 \frac{p_\alpha(\E(t(s)))}{\mathcal Y(t(s))}
 \dot t_\alpha(s)\,\dd s
 \right).
\end{aligned}
\end{equation}
Consequently, a nonzero discriminant cannot reach zero along such a path.
\end{remark}

\subsection{A transfer theorem in singularity analysis}

The reconstruction produces local Laurent--Puiseux expansions, and the last step is to transfer them to asymptotics of GW invariants. We use generating functions in the exponential variable, which are natural for GW potentials. For the coefficient-transfer principle, see \cite[Chapter~VI]{Analytic_Combinatorics}; for the local continuation formula for the polylogarithm used below, see \cite[Eq.~(25.12.12)]{NIST_DLMF}.

For $\rho>1$ and $0<\theta<\pi/2$, consider the domain 
\begin{equation}\label{p2:eq:Delta-domain}
 \mathfrak D(\rho,\theta)
 =
 \left\{
 u\in\CC:
 |u|<\rho,\quad
 u\neq1,\quad
 |\arg(u-1)|>\theta
 \right\}.
\end{equation}
We say a function has a $\mathfrak D$-continuation at $u=1$ if it is holomorphic on some domain of the form $\mathfrak D(\rho,\theta)$, where the argument in \eqref{p2:eq:Delta-domain} is the principal argument in $(-\pi,\pi]$.

\begin{theorem}\label{p2:thm:polylog-expansion}
For the polylogarithm function
\[
 \Li_s(u)=\sum_{d\geq1}\frac{u^d}{d^s},
 \qquad |u|<1,
\]
put $u=e^\mu$, where $0<|\mu|<2\pi$ and $\mu\notin\R_{\geq0}$. With $-\pi<\arg(-\mu)\leq\pi$, and with $(-\mu)^\alpha$ and $\log(-\mu)$ computed on this principal branch, one has
\begin{align}
 \Li_s(e^\mu)
 &=
 \Gamma(1-s)(-\mu)^{s-1}
 +
 \sum_{m\geq0}\zeta(s-m)\frac{\mu^m}{m!},
 &&
 s\notin\Z_{\geq1},
 \label{p2:eq:polylog-noninteger}\\
 \Li_r(e^\mu)
 &=
 \frac{\mu^{r-1}}{(r-1)!}
 \bigl(H_{r-1}-\log(-\mu)\bigr)
 +
 \sum_{\substack{m\geq0\\m\neq r-1}}
 \zeta(r-m)\frac{\mu^m}{m!},
 &&
 r\in\Z_{\geq1},
 \label{p2:eq:polylog-integer}
\end{align}
where $\zeta(s)$ is the Riemann zeta function, $H_0=0$, and $H_m=\sum_{j=1}^m j^{-1}$.
\end{theorem}

\begin{proof}
Formula \eqref{p2:eq:polylog-noninteger} is \cite[Eq.~(25.12.12)]{NIST_DLMF}, after substituting $u=e^\mu$ and using the branch conventions fixed above. Formula \eqref{p2:eq:polylog-integer} follows by taking the limit $s\to r\in\mathbb Z_{\geq1}$ in that identity and cancelling the pole of $\Gamma(1-s)$ against the pole of the zeta term with $m=r-1$.
\end{proof}

The polylogarithm formula and the Laurent--Puiseux expansions use different argument conventions in the following. For the latter, cut the $\zeta$-plane along $\R_{\geq0}$. If $\zeta=re^{\iu\vartheta}$, with $r>0$ and $0<\vartheta<2\pi$, set
\[
 \log \zeta:=\log r+\iu\vartheta,
 \qquad
 \zeta^\alpha:=\exp\bigl(\alpha\log \zeta\bigr).
\]
When $\mu=\zeta=\log u$, with the local logarithm normalized by $\log 1=0$, the two conventions are related by
\begin{equation}\label{p2:eq:branch-conversion}
\begin{aligned}
 \arg(-\mu)=\vartheta-\pi,\qquad
 \log \zeta =\log (-\mu)+\pi\iu,\qquad
 \zeta^\alpha=e^{\pi\iu\alpha}(-\mu)^\alpha.
\end{aligned}
\end{equation}
These identities extend to the two sides of $\R_{\ge 0}$ by the corresponding boundary values.

\begin{theorem}\label{p2:thm:big-oh-transfer}
Let $F(u)=\sum_{d\geq0}c_du^d$ be analytic at $u=0$ and have a $\mathfrak D$-continuation at $u=1$. Let $\alpha\in\R$, and take all estimates uniformly as $u\to1$ in that domain. Then
\[
 F(u)=O\bigl((1-u)^\alpha\bigr)
 \quad\Longrightarrow\quad
 c_d=O\bigl(d^{-\alpha-1}\bigr).
\]
The corresponding implication with $o$ in place of $O$ also holds.
\end{theorem}

\begin{proof}
This is the big-oh and little-oh transfer theorem in \cite[Theorem VI.3]{Analytic_Combinatorics}.
\end{proof}

\begin{theorem}\label{p2:thm:transfer}
Let $G(z)=\sum_{d\geq1}c_de^{dz}$ converge for $\Re z<x_*$. Suppose that, for some $\delta>0$, it continues analytically to an open domain containing
\[
 \{\Re z<x_*+\delta,\ 0\leq\Im z\leq2\pi\}
 \setminus\{x_*,x_*+2\pi\iu\}.
\]
In both endpoint expansions write $\zeta=re^{\iu\vartheta}$, using $0\leq\vartheta\leq\pi$ in the upper sector and $\pi\leq\vartheta\leq2\pi$ in the lower sector, and define
\begin{equation}\label{p2:eq:half-power-branch}
 \zeta^{1/2}:=r^{1/2}e^{\iu\vartheta/2},
 \qquad
 \zeta^{j/2}:=(\zeta^{1/2})^j.
\end{equation}
Thus the upper and lower boundary values on $\R_{>0}$ are $+\sqrt r$ and $-\sqrt r$, respectively. Assume that, uniformly as $\zeta\to0$ in the upper sector,
\begin{equation}\label{p2:eq:local-transfer}
  G(x_*+\zeta)
  =
  \sum_{\substack{j\in\Z\\j_0\leq j<2A}}b_j\zeta^{j/2}
  +
  O(\zeta^A),
\end{equation}
where $0\leq\arg\zeta\leq\pi$, and that, uniformly in the lower sector,
\begin{equation}\label{p2:eq:local-transfer-lower}
  G(x_*+2\pi\iu+\zeta)
  =
  \sum_{\substack{j\in\Z\\j_0\leq j<2A}}b_j\zeta^{j/2}
  +
  O(\zeta^A),
\end{equation}
where $\pi\leq\arg\zeta\leq2\pi$. Here $A\in\R\setminus\N$ and $j_0\in\Z$ satisfies $j_0/2<A$. The notation $O(\zeta^A)$ means $O(|\zeta|^A)$ uniformly in the indicated sectors. Then
\begin{equation}\label{p2:eq:transfer-result}
  c_d
  =
  e^{-dx_*}
  \left(
    \sum_{\substack{j\in\Z,\ j_0\leq j<2A\\j\notin2\N}}
    \frac{\iu^j b_j}{\Gamma(-j/2)}
    d^{-j/2-1}
    +
    O(d^{-A-1})
  \right).
\end{equation}
The condition $j\notin2\N$ removes nonnegative even $j$. If the local expansions are convergent half-integer Laurent series with finite principal part, then \eqref{p2:eq:transfer-result} may be truncated at arbitrary order.
\end{theorem}

\begin{proof}
Near the image of $x_*$, define $F(u)=G(x_*+\log u)$, and near the image of $x_*+2\pi\iu$, define $F(u)=G(x_*+2\pi\iu+\log u)$, where $\log 1=0$. For $|u|<1$ sufficiently close to $1$, both expressions equal $\sum_{d\geq1}e^{dx_*}c_du^d$. Therefore, the two expressions glue to a single-valued $\mathfrak D$-continuation at $u=1$.

Put $\mu=\log u$. If $\alpha\in\R\setminus\N$, the relations of branch conventions \eqref{p2:eq:branch-conversion} and Theorem~\ref{p2:thm:polylog-expansion} give
\begin{equation*}
  \zeta^\alpha
  =
  \frac{e^{\pi\iu\alpha}}{\Gamma(-\alpha)}
  \Li_{\alpha+1}(u)
  +
  H_\alpha(u),
\end{equation*}
where $H_\alpha$ is holomorphic at $u=1$. For $\alpha\in\N$, the function $\zeta^\alpha=(\log u)^\alpha$ is holomorphic at $u=1$ and is absorbed into the holomorphic Taylor part. Taking $\alpha=j/2$ in the remaining terms gives $e^{\pi\iu\alpha}=\iu^j$. Since $[u^d]\Li_s(u)=d^{-s}$, the singular terms give the sum in \eqref{p2:eq:transfer-result}. Collect the finitely many holomorphic terms in $H(u)$, choose an integer $M>\max\{A,0\}$, and subtract the Taylor polynomial of $H$ at $u=1$ of degree $M-1$ in $1-u$. The remaining holomorphic term is $O((1-u)^M)=o((1-u)^A)$, whereas the Taylor polynomial has only finitely many coefficients. Since $\log u=-(1-u)+O((1-u)^2)$, the remainder is $O((1-u)^A)$ in a smaller fixed $\mathfrak D$-domain. Theorem~\ref{p2:thm:big-oh-transfer} proves the error estimate. Successive truncations prove the last assertion.
\end{proof}

\section{Genus zero}\label{p2:sec:g0}

We now specialize the abstract framework to the quantum cohomology of $\P{2}$, derive the one-variable differential equation, and identify its dominant singularity.

Write an element of $H^*(\P{2},\CC)$ as $t=t_0\one+t_1H+t_2H^2$. In the flat basis $(\one,H,H^2)$, the Poincar\'e pairing is $\eta_{\alpha\beta}=\delta_{\alpha+\beta,2}$. The primary GW potentials are
\begin{equation}\label{p2:eq:potentials}
 F_g(t)=
 \begin{cases}
 \displaystyle
 \frac12t_0^2t_2+\frac12t_0t_1^2+
 \sum_{d\geq1}\frac{N_{0,d}}{(3d-1)!}
 e^{dt_1}t_2^{3d-1}Q^d,&g=0,\\[5pt]
 \displaystyle
 -\frac18t_1+
 \sum_{d\geq1}\frac{N_{1,d}}{(3d)!}
 e^{dt_1}t_2^{3d}Q^d,&g=1,\\[5pt]
 \displaystyle
 \sum_{d\geq1}\frac{N_{g,d}}{(3d-1+g)!}
 e^{dt_1}t_2^{3d-1+g}Q^d,&g\geq2.
 \end{cases}
\end{equation}

\begin{remark}
We justify the degree-zero terms in \eqref{p2:eq:potentials}. In the stable range $2g-2+n>0$, the constant-map formula is \cite{Hodge_integrals_and_Gromov-Witten_theory}
\[
 \Mbar_{g,n}(\P{2},0)
 \cong
 \Mbar_{g,n}\times\P{2},
 \qquad
 [\Mbar_{g,n}(\P{2},0)]^{\mathrm{vir}}
 =
 c_{2g}(\mathbb E^\vee\boxtimes T_{\P{2}})
 \cap[\Mbar_{g,n}\times\P{2}].
\]
Here $\mathbb E=\pi_*\omega_\pi$ is the rank-$g$ Hodge bundle on $\Mbar_{g,n}$ for the universal curve $\pi$, and $\lambda_i=c_i(\mathbb E)$.
In genus zero this gives the classical cup-product potential $\frac12t_0^2t_2+\frac12t_0t_1^2$. In genus one,
\[
 c_2(\mathbb E^\vee\boxtimes T_{\P{2}})
 =
 c_2(T_{\P{2}})
 -\lambda_1c_1(T_{\P{2}})
 +\lambda_1^2.
\]
Consequently, by \cite[Theorem~1.1]{On_degree_zero_elliptic_orbifold_Gromov-Witten_invariants},
\[
 \lang{H}_{1,1,0}
 =
 -\left(\int_{\Mbar_{1,1}}\lambda_1\right)
  \left(\int_{\P{2}}Hc_1(T_{\P{2}})\right)
 =
 -\frac18.
\]
The degree-zero term in genus one follows. For $g\geq2$, the string and divisor equations reduce the calculation to invariants with $n$ insertions of $H^2$. The dimension equation is  
\[
 2n
 =
 \operatorname{vdim}_{\CC}\Mbar_{g,n}(\P{2},0)
 =
 g-1+n,
\]
so $n=g-1$. If $g\geq3$, $H^{2n}$ vanishes on the $\P{2}$-factor. In the remaining case $(g,n)=(2,1)$, the target-degree-zero part of the obstruction Euler class is $\lambda_2^2$, and
\[
 \lang{H^2}_{2,1,0}
 =
 \int_{\Mbar_{2,1}}\lambda_2^2
 =
 0
\]
because the Hodge bundle is pulled back from the three-dimensional stack $\Mbar_{2,0}$. Then the degree-zero term for $g\ge 2$ follows.
\end{remark}

The Euler vector field and Hodge grading operator are
\begin{equation*}
 E=t_0\partial_{t_0}+3\partial_{t_1}-t_2\partial_{t_2},
 \qquad
 \mu=\diag(-1,0,1).
\end{equation*}
At $t_0=t_2=0$, the matrix of $E*$ is
\begin{equation}\label{p2:eq:small-quantum}
 \E|_{t_0=t_2=0}=
 \begin{pmatrix}
 0&0&3Qe^{t_1}\\
 3&0&0\\
 0&3&0
 \end{pmatrix}.
\end{equation}
Its eigenvalues are distinct when $Qe^{t_1}\neq0$, so the quantum cohomology Frobenius manifold is generically semisimple.

The WDVV equation for $\P{2}$ is
\begin{equation}\label{p2:eq:plane-wdvv}
 F_{112}^2=F_{111}F_{122}+F_{222}.
\end{equation}
Consider the restriction to
\begin{equation}
 S:=\{t_0=0,\ t_1=z,\ t_2=1,\ Q=1\}.
\end{equation}
In what follows, the notation $F_g(z)$, with $z$ a complex variable, denotes the positive-degree series
\begin{equation}
  F_g(z) := \sum_{d\geq1}
  \frac{N_{g,d}}{(3d-1+g)!}e^{dz},
\end{equation}
as in \eqref{p2:eq:Fg-definition}, whereas $F_g(t)$, with $t=(t_0,t_1,t_2)$, denotes the GW potential. The restrictions of the GW potentials in \eqref{p2:eq:potentials} are
\[
 F_0(t)|_S=F_0(z),\qquad
 F_1(t)|_S=-\frac z8+F_1(z),\qquad
 F_g(t)|_S=F_g(z)\quad(g\geq2).
\]
Then
\begin{align}
 F_{111}|_S&=F_0''',&
 F_{112}|_S&=3F_0'''-F_0'',\nonumber\\
 F_{122}|_S&=9F_0'''-9F_0''+2F_0',&
 F_{222}|_S&=27F_0'''-54F_0''+33F_0'-6F_0.
\label{p2:eq:derivative-table}
\end{align}
Substitution into \eqref{p2:eq:plane-wdvv} gives
\begin{equation}\label{p2:eq:scalar-wdvv}
 (27+2F_0'-3F_0'')F_0'''
 =
 6F_0-33F_0'+54F_0''+(F_0'')^2.
\end{equation}
Put
\begin{equation}\label{p2:eq:XY}
 \mathcal X:=6F_0-33F_0'+54F_0''+(F_0'')^2,
 \qquad
 \mathcal Y|_S=27+2F_0'-3F_0''.
\end{equation}
A direct determinant calculation from \eqref{p2:eq:Y-definition} gives the restriction of the cyclic determinant $\mathcal Y$. From now on, when the restriction to the one-variable slice $S$ is understood, we suppress the symbol ${}|_S$ and write it simply as $\mathcal Y$. Comparing coefficients in \eqref{p2:eq:scalar-wdvv} gives Kontsevich's recursion with $N_{0,1}=1$; see also \cite{Quantum_intersection_rings}. The estimates proved in \cite[Lemma~2.1]{Asymptotic_of_enumerative_invariants_in_CP2} give
\begin{equation}\label{p2:eq:coarse-bound}
 27^{-d}d^{-7/2}
 \leq
 \frac{N_{0,d}}{(3d-1)!}
 \leq
 3(4/15)^d d^{-7/2}.
\end{equation}
In particular, every $N_{0,d}$ is positive. Consequently, $F_0$ has a finite real abscissa of convergence $x_0$, with $\log(15/4)\leq x_0\leq\log27$.

The following proposition is a sharpened form of \cite[Lemma~3.1]{Asymptotic_of_enumerative_invariants_in_CP2}; we include the proof in order to fix the branch convention and to compute the coefficient $a_5$ explicitly.

\begin{proposition}\label{p2:prop:fold}
On the interval $(-\infty,x_0)$, we have
\[
 \mathcal X>0,\qquad \mathcal Y>0.
\]
The functions $F_0,F_0',F_0''$ have finite limits at $x_0$, and
\[
 \mathcal Y(x_0)=0,\qquad \mathcal X(x_0)>0.
\]
The function $F_0$ admits a locally convergent expansion on the upper side of a slit neighborhood of $x_0$. Write $\zeta=z-x_0$ and compute its powers with the upper branch in \eqref{p2:eq:half-power-branch}. Then
\begin{equation}\label{p2:eq:fold-expansion}
 F_0(x_0+\zeta)
 =
 F_0(x_0)+F_0'(x_0)\zeta+\frac12F_0''(x_0)\zeta^2
 +a_5\zeta^{5/2}
 +\sum_{j\geq6}a_j\zeta^{j/2},
\end{equation}
where
\begin{equation}\label{p2:eq:a5}
 a_5
 =
 \frac{4\iu}{15}
 \sqrt{\frac{2\mathcal X(x_0)}3}
 =
 \iu\sqrt{
 \frac{
  32\bigl(4F_0'(x_0)^2+54F_0(x_0)+135F_0'(x_0)+5103\bigr)
 }{6075}
 }
 \in\iu\R_{>0}.
\end{equation}
The lower-side expansion at $x_0+2\pi\iu$ uses the same coefficients $a_j$, computed with the lower branch \eqref{p2:eq:half-power-branch}; this follows from periodicity.
\end{proposition}

\begin{proof}
Positivity of the coefficients gives $0<F_0<F_0'<F_0''<F_0'''$ on the real convergence interval. The coefficient of $e^{dz}$ in $6F_0-33F_0'+54F_0''$ is $(6-33d+54d^2)N_{0,d}/(3d-1)!$, hence $\mathcal X>0$. Equation \eqref{p2:eq:scalar-wdvv} then gives $\mathcal Y>0$, and therefore
\[
 0<F_0<F_0'<F_0''<3F_0''-2F_0'<27.
\]
The first three functions have finite values at $x_0$. We claim that
\begin{equation}
  \mathcal Y(x_0) = 27+2F_0'(x_0)-3F_0''(x_0) =0.
\end{equation}
Indeed, if $\mathcal Y(x_0)\neq0$, the analytic ordinary differential equation for $(F_0,F_0',F_0'')$ would continue the solution through $x_0$. This contradicts Pringsheim's theorem applied to the positive power series in $q=e^z$. Thus $\mathcal Y(x_0)=0$, while the positivity gives $\mathcal X(x_0)>0$.

Set $u=F_0$, $v=\partial_zF_0$, and $w=\partial_z^2F_0$, and introduce a new variable $\tau$ by
\[
 \frac{\dd\tau}{\dd z}=\frac1{\mathcal Y}.
\]
Along the real solution,
\[
 \frac{\dd w}{\dd z}=\frac{\mathcal X}{\mathcal Y},
 \qquad
 \frac{\dd\tau}{\dd w}=\frac1{\mathcal X}.
\]
Since $w$ has a finite limit at $x_0$ and $\mathcal X\to\mathcal X(x_0)>0$, the function $\tau$ has a finite limit $\tau_0$ as $z\uparrow x_0$. In the parameter $\tau$ the equation becomes
\begin{equation}\label{p2:eq:desingularized-system}
 \frac{\dd}{\dd\tau}
 \begin{pmatrix}u\\v\\w\\z\end{pmatrix}
 =
 \begin{pmatrix}
 v\mathcal Y\\w\mathcal Y\\\mathcal X\\\mathcal Y
 \end{pmatrix}.
\end{equation}
It extends holomorphically through $\tau_0$. At that point,
\[
 \frac{\dd z}{\dd\tau}(\tau_0)=0,\qquad
 \frac{\dd^2z}{\dd\tau^2}(\tau_0)
 =-3\mathcal X(x_0)<0,\qquad
 \frac{\dd w}{\dd\tau}(\tau_0)=\mathcal X(x_0)>0.
\]
Inverting the convergent Taylor series expresses $\tau-\tau_0$ as a convergent series in $(z-x_0)^{1/2}$. Since
\[
 \frac{\dd u}{\dd\tau}
 =v\frac{\dd z}{\dd\tau},
 \qquad
 \frac{\dd v}{\dd\tau}
 =w\frac{\dd z}{\dd\tau},
\]
the terms of orders $1/2$ and $3/2$ in $u$ vanish. Moreover,
\[
 F_0''(x_0+\zeta)
 =
 F_0''(x_0)
 +\iu\sqrt{\frac{2\mathcal X(x_0)}3}\,\zeta^{1/2}
 +O(\zeta).
\]
Integrating twice proves \eqref{p2:eq:fold-expansion} and \eqref{p2:eq:a5}.  Periodicity gives the expansion at $x_0+2\pi\iu$ with the same coefficients.
\end{proof}

\begin{remark}\label{p2:rem:puiseux-recursion}
Set
\[
 a_0:=F_0(x_0),\qquad
 a_1=a_3:=0,\qquad
 a_2:=F_0'(x_0),\qquad
 a_4:=\frac12F_0''(x_0).
\]
Substitution of \eqref{p2:eq:fold-expansion} into
\eqref{p2:eq:scalar-wdvv} determines every $a_k$, $k\geq6$,
recursively:
\begin{align*}
 45k(k-2)(k-3)a_5a_k
 ={}&
 -192a_{k-5}
 +528(k-3)a_{k-3}\\
 &+
 16(k-1)(k-3)\bigl((k-7)a_4-27\bigr)a_{k-1}\\
 &+
 \sum_{\substack{k_1+k_2=k-7\\k_1,k_2\geq0}}
 \Bigl[
 2(k_1+5)(k_2+5)(k_2+3)(2k_2-k_1-1)\\
 &\hspace{23mm}\cdot a_{k_1+5}a_{k_2+5}\\
 &\hspace{7mm}
 -3(k_1+6)(k_2+6)(k_1+4)(k_2+4)(k_2+2)\\
 &\hspace{23mm}\cdot a_{k_1+6}a_{k_2+6}
 \Bigr].
\end{align*}
Induction gives $a_{2m}\in\R$ and $a_{2m-1}\in\iu\R$ for $m\geq1$.
\end{remark}

We next show that, modulo the $2\pi\iu$-periodicity, $x_0$ is the unique dominant singularity. Indeed, if $\Re z<x_0$, then
\[
 \abs{3F_0''(z)-2F_0'(z)}
 \leq
 \sum_{d\geq1}(3d^2-2d)
 \frac{N_{0,d}}{(3d-1)!}e^{d\Re z}
 <
 27.
\]
Note that the series for $F_0$, $F_0'$, and $F_0''$ converge absolutely and uniformly on $\{x_0+\iu y:0\leq y\leq2\pi\}$. For $z=x_0+\iu y$ with $0<y<2\pi$, 
\begin{align}
 \abs{3F_0''(z)-2F_0'(z)}
 <
 \sum_{d\geq1}(3d^2-2d)
 \frac{N_{0,d}}{(3d-1)!}e^{dx_0} =3F_0''(x_0)-2F_0'(x_0)=27.
\label{p2:eq:boundary-strict}
\end{align}
Hence $\mathcal Y\neq0$ in the half-strip $\{\Re z\leq x_0,\ 0\leq\Im z\leq2\pi\} \setminus\{x_0,x_0+2\pi\iu\}$. By \eqref{p2:eq:scalar-wdvv}, it follows that the function $F_0(z)$ can be analytically continued locally around any point in the vertical segment $\{x_0+\iu y:0 < y < 2\pi\}$. Proposition~\ref{p2:prop:fold} gives the two endpoint neighborhoods. After smaller endpoint neighborhoods are removed, the remaining part of the vertical segment is compact. A finite subcover of the preceding ODE neighborhoods, together with the original left half-plane, therefore gives a single $\delta>0$ for which $F_0$ continues analytically to an open domain containing
\[
 \{\Re z<x_0+\delta,\ 0\leq\Im z\leq2\pi\}
 \setminus\{x_0,x_0+2\pi\iu\}.
\]
Therefore, $F_0(z)$ satisfies the conditions in Theorem~\ref{p2:thm:transfer}.

\begin{corollary}\label{p2:cor:g0}
For the genus-zero case of $\P{2}$, the expansion \eqref{p2:eq:g0-main} in Theorem~\ref{p2:thm:main} holds, and
\begin{equation}\label{p2:eq:g0-leading}
 a_{0,0}
 =
 \frac{\iu a_5}{\Gamma(-5/2)}
 =
 -\frac{15\iu a_5}{8\sqrt\pi}
 >0.
\end{equation}
\end{corollary}

\begin{proof}
Apply Theorem~\ref{p2:thm:transfer} to the convergent expansion \eqref{p2:eq:fold-expansion}. Formula \eqref{p2:eq:g0-leading} is the transferred $5/2$-term.
\end{proof}

\section{Genus one}\label{p2:sec:g1}

We now turn to the genus-one case. We reconstruct the genus-one generating function and study its singularity.

\begin{proposition}\label{p2:prop:g1}
The derivative of the genus-one generating function $F_1(z)=\sum_{d\geq1} \frac{N_{1,d}}{(3d)!}e^{dz}$ can be expressed as a rational function in the derivatives of $F_0(z)$:
\begin{equation}\label{p2:eq:g1-derivative}
 F_1'(z)
 =
 \frac18+
 \frac{F_0'''(z)-27}
 {8(27+2F_0'(z)-3F_0''(z))}.
\end{equation}
\end{proposition}

\begin{proof}
By the graph sum formula \eqref{p2:eq:graph-F1}, for $\P{2}$ we have
\begin{equation}
  \partial_{t_1}F_1(t)
 =
 \sum_{\Gamma\in G_{1,1}}
 \frac1{\abs{\Aut\Gamma}}
 \int_{\Mbar_{1,1}}\Cont_\Gamma(H).
\end{equation}
There are two stable graphs of type $(1,1)$. For the graph $\Gamma_1$ with one genus-one vertex and one leg, expanding the contribution of this graph gives
\begin{align*}
  \int_{\Mbar_{1,1}} \Cont_{\Gamma_1}(H) & = \int_{\Mbar_{1,1}}(\mathsf T\omega)_{1,1}(R^{-1}(\psi_1)H) \\
  & =\int_{\Mbar_{1,1}}(\mathsf T\omega)_{1,1}(H)-\int_{\Mbar_{1,1}}(\mathsf T\omega)_{1,1}(R_1 H)\psi_1  \\
  & =\int_{\Mbar_{1,2}}\omega_{1,2}(H,\mathsf T(\psi_2))-\int_{\Mbar_{1,1}}\omega_{1,1}(R_1 H)\psi_1     \\
  & =\omega_{1,2}(H,R_1 \one)\int_{\Mbar_{1,2}}\psi_2^2-\omega_{1,1}(R_1 H)\int_{\Mbar_{1,1}}\psi_1 \\
  & =\frac1{24} \bigl(\omega_{1,2}(H,R_1\one)-\omega_{1,1}(R_1H)\bigr).
\end{align*}
Here $\omega$ is the topological part of the shifted CohFT of $\P{2}$.

For the graph $\Gamma_2$ with one genus-zero vertex, one self-edge, and one leg, $\Mbar_{0,3}$ is a point. Thus only the constant terms of the leg and
edge factors contribute. The constant edge bivector in \eqref{p2:eq:edge-expanded} is $\sum_{\alpha=0}^2R_1H^\alpha\otimes H^{2-\alpha}$, so 
\[
 \int_{\Mbar_{1,1}} \Cont_{\Gamma_2}(H) = \sum_{\alpha=0}^2
 \omega_{0,3}(H,R_1H^\alpha,H^{2-\alpha}).
\]
The first graph $\Gamma_1$ has trivial automorphism group, while $\abs{\Aut\Gamma_2}=2$ because exchanging the two half-edges gives a nontrivial automorphism. It follows that
\begin{equation}
  \partial_{t_1}F_1(t) = 
  \frac1{24} \bigl(\omega_{1,2}(H,R_1\one)-\omega_{1,1}(R_1H)\bigr)
  + \frac12 \sum_{\alpha=0}^2
 \omega_{0,3}(H,R_1H^\alpha,H^{2-\alpha}).
\end{equation}
Consider the restriction to the locus $S=\{t_0=0,\ t_1=z,\ t_2=1,\ Q=1\}$. Recall that $F_1(t)|_S=-\frac z8+F_1(z)$. Solve $R_1|_S$ by the recursion \eqref{p2:eq:R-recursion} with 
\[
  \E|_S=
  \begin{pmatrix}
 0&6F_0''-2F_0'&27F_0''-27F_0'+6F_0\\
 3&F_0''&6F_0''-2F_0'\\
 -1&3&0
 \end{pmatrix}
  ,\qquad 
  \mu=
  \begin{pmatrix}
  -1&&\\
  &0&\\
  &&1
 \end{pmatrix}.
\]
Let $m_u$ denote Frobenius multiplication by $u$. Using
\[
 \eta(u,\Delta)=\operatorname{Tr}(m_u),
\]
formula \eqref{p2:eq:TFT} turns the preceding graph expression into
\[
 \left.\partial_{t_1}F_1(t)\right|_S
 =
 \frac1{24}
 \bigl(
   \operatorname{Tr}\left(m_Hm_{R_1\one}\right)
   -\operatorname{Tr}\left(m_{R_1H}\right)
 \bigr)
 +\frac12\operatorname{Tr}(m_HR_1),
\]
where all operators are restricted to $S$. Substituting the matrix $R_1|_S$ from Appendix~\ref{p2:app:explicit-R1-D}, and simplifying with \eqref{p2:eq:scalar-wdvv}, gives
\[
 \left.\partial_{t_1}F_1(t)\right|_S
 =
 \frac{F_0'''(z)-27}
 {8\bigl(27+2F_0'(z)-3F_0''(z)\bigr)}.
\]
Since $\left.\partial_{t_1}F_1(t)\right|_S=-\frac18+F_1'(z)$, rearranging proves \eqref{p2:eq:g1-derivative}.
\end{proof}

Substitution of \eqref{p2:eq:fold-expansion} into \eqref{p2:eq:g1-derivative} gives the convergent local Laurent--Puiseux expansion
\begin{equation}\label{p2:eq:g1-local}
 F_1'(x_0+\zeta)
 =
 -\frac1{48\zeta}+O(\zeta^{-1/2}).
\end{equation}
At $x_0+2\pi\iu$, by periodicity, the lower-branch
Laurent--Puiseux expansion has the same coefficients. Recall that $\mathcal Y=27+2F_0'(z)-3F_0''(z)\neq0$ in the half-strip $\{\Re z\leq x_0,\ 0\leq\Im z\leq2\pi\} \setminus\{x_0,x_0+2\pi\iu\}$ and that
\[
 \mathcal Y(x_0+\zeta)
 =-\frac{45}{4}a_5\zeta^{1/2}+O(\zeta),
 \qquad a_5\neq0.
\]
After decreasing the $\delta$ in the genus-zero continuation if necessary, compactness of the closed boundary segment away from the two endpoint neighborhoods, together with the endpoint expansions, shows that the denominator remains nonzero on the domain required by Theorem~\ref{p2:thm:transfer}. Thus $F_1'(z)$ satisfies that theorem's hypotheses, and $x_0$ is the genus-one abscissa of convergence.

\begin{corollary}\label{p2:cor:g1}
For the genus-one case of $\P{2}$, the expansion \eqref{p2:eq:g1-main} in Theorem~\ref{p2:thm:main} holds.
\end{corollary}

\begin{proof}
By \eqref{p2:eq:g1-derivative} and the convergent expansion \eqref{p2:eq:fold-expansion}, the function $F_1'$ has a convergent half-integer Laurent--Puiseux expansion with finite principal part at each of the two endpoint singularities. Applying Theorem~\ref{p2:thm:transfer} to successive truncations gives constants $a_{1,k}$, $k\geq1$, such that, for every $N\geq1$,
\[
 d\frac{N_{1,d}}{(3d)!}
 =
 e^{-dx_0}
 \left(
   \frac1{48}
   +\sum_{k=1}^{N-1}a_{1,k}d^{-k+1/2}
   +O(d^{-N+1/2})
 \right).
\]
Dividing by $d$ proves \eqref{p2:eq:g1-main}.
\end{proof}

\section{Higher genus}\label{p2:sec:hg}

Next, we consider the genera $g\ge 2$ for $\P{2}$. We reconstruct the genus-$g$ generating function and study its singularity.

For $g\geq2$, restrict the graph reconstruction of the potential \eqref{p2:eq:potentials} to $S=\{t_0=0,\ t_1=z,\ t_2=1,\ Q=1\}$. Proposition~\ref{p2:prop:reconstruction} gives
\begin{equation}\label{p2:eq:localized-Fg}
 F_g(z)
 \in
 \Q\left[
 F_0,F_0',F_0'',\mathcal Y^{-1},\mathcal D^{-1}
 \right],
\end{equation}
where we suppress the symbol ${}|_S$. Recall that $\mathcal{Y}=27+2F_0'(z)-3F_0''(z)$ and $\mathcal{D}$ is the discriminant of the characteristic polynomial of the matrix
\begin{equation}\label{p2:eq:euler-on-S}
 \E|_S=
 \begin{pmatrix}
 0&6F_0''-2F_0'&27F_0''-27F_0'+6F_0\\
 3&F_0''&6F_0''-2F_0'\\
 -1&3&0
 \end{pmatrix}.
\end{equation}
The discriminant $\mathcal{D}$ can be expressed as a rather lengthy polynomial in $F_0, F_0', F_0''$, whose complete expression is recorded in Appendix~\ref{p2:app:explicit-R1-D}. At the origin $q=e^z=0$,
\begin{equation}\label{p2:eq:small-q-denominators}
 \mathcal Y=27+O(q),
 \qquad
 \mathcal D=-3^9q^2+O(q^3).
\end{equation}
The graph sum therefore has a Laurent expansion in $q$ at the origin. Reconstruction identifies it coefficientwise with the formal series $F_g\in q\CC[[q]]$. Hence all negative Laurent coefficients vanish, and $F_g(q)$ is analytic at the origin.

\begin{lemma}\label{p2:lem:discriminant}
Along $S=\{t_0=0,\ t_1=z,\ t_2=1,\ Q=1\}$, 
\begin{equation}\label{p2:eq:discriminant-ode}
 \mathcal D'
 =
 \frac{2(2F_0''+27)}{\mathcal Y}\mathcal D.
\end{equation}
The discriminant $\mathcal{D}$ extends to the half-strip $\{\Re z\leq x_0,\ 0\leq\Im z\leq2\pi\}$, and it has no zeros in this half-strip. In particular, $\mathcal{D}(x_0)\neq 0$.
\end{lemma}

\begin{proof}
Equation \eqref{p2:eq:discriminant-ode} follows from the explicit computation of \eqref{p2:eq:D-propagation} using the algorithm in the proof of Proposition~\ref{p2:prop:D-propagation}. It can also be checked by direct differentiation of the explicit formula for $\mathcal D$ in Appendix~\ref{p2:app:explicit-R1-D}.

By \eqref{p2:eq:small-q-denominators}, choose a negative real $z_0$ such that $\mathcal D(z_0)\neq0$. The punctured half-strip $\{\Re z\leq x_0,\ 0\leq\Im z\leq2\pi\}\setminus\{x_0,x_0+2\pi\iu\}$ is path connected, and $\mathcal Y$ has no zero there. Hence, for every finite point $z$ of that half-strip away from the two endpoints $\{x_0, x_0+2\pi\iu\}$, a path from $z_0$ to $z$ can be chosen on which $\mathcal Y\neq0$. Along this path,
\[
 \mathcal D(z)
 =
 \mathcal D(z_0)
 \exp\left(
   \int_{z_0}^{z}
   \frac{2(2F_0''(s)+27)}{\mathcal Y(s)}\,\dd s
 \right),
\]
so $\mathcal D(z)\neq0$.

The convergent Laurent--Puiseux expansions of $F_0,F_0',F_0''$ show that the polynomial $\mathcal D$ has finite limits at the two endpoints. Moreover,
\[
 \mathcal Y(x_0+\zeta)
 =-\frac{45}{4}a_5\zeta^{1/2}+O(\zeta),
\]
and hence the logarithmic derivative above is $O(\zeta^{-1/2})$. Its integral converges as $z\to x_0$, so the exponential has a finite nonzero limit. Thus $\mathcal D(x_0)\neq0$. Periodicity gives the same conclusion at $x_0+2\pi\iu$.
\end{proof}

\begin{remark}
The nonvanishing $\mathcal D(x_0)\neq0$ means that the eigenvalues of the limiting matrix $\E|_S$ remain pairwise distinct as $z\to x_0$ in the slit neighborhood. Thus, the singular behavior is not caused by a collision of these limiting eigenvalues; it occurs in the change from canonical to flat coordinates. This is consistent with the local quantum-cohomology analysis in \cite{The_singularity_of_Kontsevichs_solution_for_QHCP2}.
\end{remark}

\begin{lemma}\label{p2:lem:R-bounded}
For every fixed $n\geq0$, every flat-basis entry of the $R$-matrix $R_n|_S$ has a convergent half-integer expansion and is $O(1)$ at $x_0$.
\end{lemma}

\begin{proof}
The entries of $\E|_S$ have convergent half-integer expansions and finite limits at $x_0$. In the linear system obtained from \eqref{p2:eq:R-recursion}, or equivalently \eqref{p2:eq:An-system}, by \eqref{p2:eq:R-system-determinant}, the determinant at level $n$ is $-(n+1)^3\mathcal D^2$, which is nonzero at $x_0$. Induction from $R_0=\id$ proves the lemma.
\end{proof}

\begin{lemma}\label{p2:lem:TFT-bound}
Let $\omega$ be the topological part of the shifted CohFT of $\P{2}$. Let $2g-2+n>0$, and let every insertion $v_i$ be the restriction to $S$ of a vector $R_{j_i}^*(H^{k_i})$, with $j_i\geq0$ and $k_i\in\{0,1,2\}$ fixed. Then
\begin{equation}\label{p2:eq:TFT-bound}
 \omega_{g,n}(v_1,\ldots,v_n)|_S(x_0+\zeta)
 =
 O\bigl(\zeta^{-(2g-2+n)/2}\bigr).
\end{equation}
\end{lemma}

\begin{proof}
Compute the TFT on a maximally degenerate stable curve. It has $2g-2+n$ trivalent genus-zero vertices. Every trivalent tensor $\omega_{0,3}|_S$ is a linear combination of third derivatives in \eqref{p2:eq:derivative-table}, each of which is $O(\zeta^{-1/2})$. Lemma~\ref{p2:lem:R-bounded} bounds the entries of the $R$-matrix, and contractions use the constant Poincar\'e pairing $\eta$. Multiplying the vertex bounds gives \eqref{p2:eq:TFT-bound}.
\end{proof}

\begin{lemma}\label{p2:lem:unit-column}
On the slice $S=\{t_0=0,\ t_1=z,\ t_2=1,\ Q=1\}$, as $\zeta\to0$ in the upper-side slit neighborhood of $x_0$, with the half-powers taken on the branch \eqref{p2:eq:half-power-branch}, 
\begin{equation}\label{p2:eq:unit-cube}
 (R_1\one)^{*3}|_S(x_0+\zeta)
 =
 -\frac{512}{50625a_5^4}
 \zeta^{-1}(3\cdot\one+H)
 +O(\zeta^{-1/2}).
\end{equation}
\end{lemma}

\begin{proof}
For $\mathbf p=p_0\one+p_1H+p_2H^2$, define $\ell(\mathbf p)=p_1+3p_2$. Put
\[
 \mathbf r:=R_1\one|_S,\qquad
 \mathbf v:=3\cdot\one+H,\qquad
 \sigma:=\frac{15a_5}{8},
\]
and set
\[
 A:=F_0(z),\qquad 
 B:=F_0'(z),\qquad 
 C:=F_0''(z).
\]
By Lemma~\ref{p2:lem:R-bounded}, the limit
\[
  \mathbf r^{(0)}
  :=
  \lim_{\substack{\zeta\to0\\0\leq\arg\zeta\leq\pi}}
  \mathbf r(x_0+\zeta)
\]
exists, and
\[
  \mathbf r(x_0+\zeta)
  =
  \mathbf r^{(0)}+O(\zeta^{1/2}).
\]
Set
\[
  A_0:=F_0(x_0),\qquad
  B_0:=F_0'(x_0),\qquad
  C_0:=F_0''(x_0).
\]
Since $\mathcal Y(x_0)=0$, we have $C_0=9+2B_0/3$. Evaluating the explicit formula for $R_1|_S$ in Appendix~\ref{p2:app:explicit-R1-D} at $(A,B,C)=(A_0,B_0,C_0)$ gives
\begin{align*}
  \ell(\mathbf r^{(0)})
  &=
  \left.
  \frac{-18\mathcal P_{21}-9\mathcal P_{31}}{\mathcal D}
  \right|_{(A,B,C)=(A_0,B_0,C_0)}\\
  &=
  \frac{27}{4B_0^2+54A_0+135B_0+5103}\\
  &=
  \frac3{\mathcal X(x_0)}.
\end{align*}
Equation~\eqref{p2:eq:a5} yields
\[
  a_5^2=-\frac{32}{675}\mathcal X(x_0),
  \qquad
  \ell(\mathbf r^{(0)})=-\frac{32}{225a_5^2}.
\]
The leading third derivatives in \eqref{p2:eq:derivative-table} give the following uniform estimate. Fix a norm $\lVert\cdot\rVert$ on the flat-coordinate space and a closed subsector of the upper-side slit neighborhood of $x_0$. There is a constant $C>0$ such that, for all sufficiently small $\zeta$ in this subsector and all vectors $\mathbf p,\mathbf q$,
\begin{equation}\label{p2:eq:rank-one-product}
  \left\lVert
  \mathbf p*\mathbf q
  -
  \sigma\zeta^{-1/2}
  \ell(\mathbf p)\ell(\mathbf q)\mathbf v
  \right\rVert
  \leq
  C\lVert\mathbf p\rVert\,\lVert\mathbf q\rVert.
\end{equation}
Applying this estimate and using $\ell(\mathbf v)=1$ gives
\begin{align*}
  \mathbf r^{*3}
  &=
  \sigma^2\zeta^{-1}\ell(\mathbf r^{(0)})^3\mathbf v
  +O(\zeta^{-1/2})\\
  &=
  \left(\frac{15a_5}{8}\right)^2
  \left(-\frac{32}{225a_5^2}\right)^3
  \zeta^{-1}\mathbf v
  +O(\zeta^{-1/2})\\
  &=
  -\frac{512}{50625a_5^4}
  \zeta^{-1}(3\cdot\one+H)
  +O(\zeta^{-1/2}).
\end{align*}
This is \eqref{p2:eq:unit-cube}.
\end{proof}

\begin{theorem}\label{p2:thm:hg-local}
For every $g\geq2$, the genus-$g$ generating function 
\begin{equation}
  F_g(z) = \sum_{d\geq1}
  \frac{N_{g,d}}{(3d-1+g)!}e^{dz}
\end{equation}
has a convergent half-integer Laurent--Puiseux expansion at $x_0$:
\begin{equation}\label{p2:eq:hg-local}
 F_g(x_0+\zeta)
 =
 \frac{\lang{\tau_2^{3g-3}}_g}
 {(3g-3)!(-15a_5)^{g-1}}
 \zeta^{-(5g-5)/2}
 +
 O\bigl(\zeta^{-(5g-6)/2}\bigr).
\end{equation}
At $x_0+2\pi\iu$, it has the same Laurent--Puiseux coefficients. Here 
\begin{equation*}
  \lang{\tau_2^{3g-3}}_g
  :=
  \int_{\Mbar_{g,3g-3}}
  \psi_1^2\cdots\psi_{3g-3}^2
\end{equation*}
is the intersection number of $\psi$-classes on the moduli space of stable curves.
\end{theorem}

\begin{proof}
We use the graph-sum formula \eqref{p2:eq:graph-Fg} and restrict to $S$. Fix a stable graph $\Gamma$ of type $(g,0)$. A typical product of vertex factors has the form
\[
  \prod_v \frac{1}{m_v!}
  \int_{\Mbar_{g_v,n_v+m_v}}
  \left(\prod_{i=1}^{n_v}\psi_i^{\alpha_{v,i}}\right)
  \left(\prod_{j=1}^{m_v}\psi_{n_v+j}^{\beta_{v,j}}\right)
  \omega_{g_v,n_v+m_v}
  (\ldots,\mathsf T_{\beta_{v,1}},\ldots,\mathsf T_{\beta_{v,m_v}}).
\]
The omitted ordinary and edge-half-edge inputs are of the form $R_i^*(H^j)$. Here, at a vertex $v$, $m_v$ is the number of translation markings, $n_v$ the number of half-edges, and $\beta_{v,j}\geq2$. The dimension constraint is 
\[
 \sum_i\alpha_{v,i}
 +
 \sum_j(\beta_{v,j}-1)
 =
 3g_v-3+n_v.
\]
Lemmas~\ref{p2:lem:R-bounded} and~\ref{p2:lem:TFT-bound} show that the contribution of this graph is
\[
 O\left(
 \zeta^{-\frac12\sum_v(2g_v-2+n_v+m_v)}
 \right).
\]
Moreover,
\begin{align*}
 \sum_v(2g_v-2+n_v+m_v)
 &\leq
 \sum_v\left(
 5g_v-5+2n_v-\sum_i\alpha_{v,i}
 \right)\\
 &\leq
 5\sum_v g_v-5\abs{V(\Gamma)}+2\sum_v n_v\\
 &=
 5g-5-\abs{E(\Gamma)},
\end{align*}
where we used $g=1-\abs{V(\Gamma)}+\abs{E(\Gamma)}+\sum_vg_v$ and $\sum_vn_v=2\abs{E(\Gamma)}$. Here $\abs{V(\Gamma)}$ and $\abs{E(\Gamma)}$ denote the number of vertices and edges of $\Gamma$, respectively. Consequently, for each fixed genus $g$ and each fixed graph $\Gamma$, uniformly as $\zeta\to0$ in any fixed closed subsector of the slit neighborhood, the contribution of $\Gamma$ is
\begin{equation}\label{p2:eq:graph-bound}
 O\bigl(
   \zeta^{-(5g-5-\abs{E(\Gamma)})/2}
 \bigr),
\end{equation}
where the implied constant may depend on $g$, $\Gamma$, and the chosen subsector. Thus, every stable graph with an edge is less singular by at least $1/2$ than the one-vertex leading term.

For the one-vertex graph $\Gamma_1$ of type $(g,0)$, its contribution is
\begin{align*}
  \int_{\Mbar_{g,0}}\Cont_{\Gamma_1}= & \int_{\Mbar_{g,0}}(\mathsf T\omega)_{g,0} \\
  = & \int_{\Mbar_{g,0}}\omega_{g,0}+\int_{\Mbar_{g,1}}\omega_{g,1}(\mathsf T(\psi_1))+\frac{1}{2!}\int_{\Mbar_{g,2}}\omega_{g,2}(\mathsf T(\psi_1),\mathsf T(\psi_2))+\cdots \\
  & +\frac{1}{(3g-3)!}\int_{\Mbar_{g,3g-3}}\omega_{g,3g-3}(\mathsf T(\psi_1),\ldots,\mathsf T(\psi_{3g-3})).
\end{align*}
Only the last term in this summation can achieve the power $-(5g-5)/2$ of $\zeta$, and in this case the $3g-3$ translations are all equal to $\mathsf T_2=R_1\one$. The contribution of this term is
\[
  \frac{1}{(3g-3)!}\omega_{g,3g-3}(R_1\one,\ldots,R_1\one)\int_{\Mbar_{g,3g-3}}\psi_1^2\cdots\psi_{3g-3}^2=\frac{\lang{\tau_2^{3g-3}}_g}{(3g-3)!} \eta\bigl((R_1\one)^{*(3g-3)},\Delta^{*g}\bigr).
\]
It remains to compute the term $\eta\bigl((R_1\one)^{*(3g-3)},\Delta^{*g}\bigr)$ restricted to $S$. Recall the computations in Lemma~\ref{p2:lem:unit-column}. Set $c:=-512/(50625a_5^4)$. By \eqref{p2:eq:unit-cube} and \eqref{p2:eq:derivative-table},
\[
 (R_1\one)^{*3}
 =c\zeta^{-1}\mathbf v+O(\zeta^{-1/2}),
 \qquad
 \Delta
 =\sum_{\alpha=0}^2H^\alpha*H^{2-\alpha}
 =\sigma\zeta^{-1/2}\mathbf v+O(1).
\]
An induction on the number of factors, using the uniform bilinear estimate \eqref{p2:eq:rank-one-product}, gives
\begin{align*}
 (R_1\one)^{*(3g-3)}
 &=c^{g-1}\sigma^{g-2}
   \zeta^{-(3g-4)/2}\mathbf v
   +O\bigl(\zeta^{-(3g-5)/2}\bigr),\\
 \Delta^{*g}
 &=\sigma^{2g-1}
   \zeta^{-(2g-1)/2}\mathbf v
   +O\bigl(\zeta^{-(2g-2)/2}\bigr).
\end{align*}
Since $\eta(\mathbf v,\mathbf v)=1$ and $c\sigma^3=-1/(15a_5)$,
\begin{align*}
 \eta\bigl((R_1\one)^{*(3g-3)},\Delta^{*g}\bigr)
 &=(c\sigma^3)^{g-1}\zeta^{-(5g-5)/2}
   +O\bigl(\zeta^{-(5g-6)/2}\bigr)\\
 &=\frac{1}{(-15a_5)^{g-1}}\zeta^{-(5g-5)/2}
   +O\bigl(\zeta^{-(5g-6)/2}\bigr),
\end{align*}
which proves \eqref{p2:eq:hg-local}.

At $x_0+2\pi\iu$, by periodicity, the lower-branch Laurent--Puiseux expansion has the same coefficients. 
\end{proof}

\begin{remark}\label{p2:rem:intersection-recursion}
Equations~\cite[(6.9)--(6.10)]{Combinatorics_of_the_modular_group_II_the_Kontsevich_integrals} give the following normalization and recursion for the intersection numbers:
\begin{align}
 \lang{\tau_2^{3g-3}}_g
 &=
 \frac{2^g(3g-3)!}{(5g-5)(5g-3)}\phi_g,
 &g\geq2,
 \label{p2:eq:tau2-phi}\\
 \phi_{g+1}
 &=
 \frac{25g^2-1}{24}\phi_g
 +
 \frac12\sum_{m=1}^{g}\phi_{g+1-m}\phi_m,
 &g\geq0,
 \label{p2:eq:phi-recursion}\\
 \phi_0&=-1.\nonumber
\end{align}
In particular, $\lang{\tau_2^{3g-3}}_g >0$ for every $g\geq2$. The first three values are
\[
 \lang{\tau_2^3}_2=\frac7{240},
 \qquad
 \lang{\tau_2^6}_3=\frac{1225}{144},
 \qquad
 \lang{\tau_2^9}_4=\frac{1816871}{48}.
\]
Moreover, \cite[Corollary~5.7]{Recursions_and_asymptotics_of_intersection_numbers} gives its large-genus asymptotic expansion 
\begin{align*}
  \lang{\tau_2^{3g-3}}_g = & \left(\frac{25}{24}\right)^g \frac{2^{g-1} \sqrt{3/5}(3g-3)!((g-1)!)^2}{\pi^2(5g-5)(5g-3)}\left(1-\frac{49}{3750 g^3}-\frac{49}{1250 g^4}+O(g^{-5})\right).
\end{align*}
\end{remark}

Recall that $\mathcal Y=27+2F_0'-3F_0''\neq0$ in the half-strip $\{\Re z\leq x_0,\ 0\leq\Im z\leq2\pi\}\setminus\{x_0,x_0+2\pi\iu\}$ and that
\[
 \mathcal Y(x_0+\zeta)
 =-\frac{45}{4}a_5\zeta^{1/2}+O(\zeta),
 \qquad a_5\neq0.
\]
The discriminant $\mathcal D$ has no zero in this half-strip. After decreasing the $\delta$ in the genus-zero continuation if necessary, compactness, the endpoint expansions, and the rational expression \eqref{p2:eq:localized-Fg} give, for every fixed $g\geq2$, the open continuation domain required by Theorem~\ref{p2:thm:transfer}. Hence $F_g(z)$ satisfies that theorem's hypotheses, and $x_0$ is its abscissa of convergence.

\begin{proof}[Proof of Theorem~\ref{p2:thm:main}]
Corollaries~\ref{p2:cor:g0} and~\ref{p2:cor:g1} prove the first two parts. For $g\geq2$, apply Theorem~\ref{p2:thm:transfer} to \eqref{p2:eq:hg-local} and to successive truncations of the convergent Laurent--Puiseux expansion. The leading power $\zeta^{-\frac{5g-5}{2}}$ transfers to $d^{(5g-7)/2}$, and its coefficient is
\[
 a_{g,0}=\frac{\iu^{-(5g-5)}
 \lang{\tau_2^{3g-3}}_g}
 {(3g-3)!(-15a_5)^{g-1}
 \Gamma((5g-5)/2)}
 =
 \frac{\lang{\tau_2^{3g-3}}_g}
 {(3g-3)!(-15\iu a_5)^{g-1}
 \Gamma((5g-5)/2)}.
\]
Since the intersection number $\lang{\tau_2^{3g-3}}_g>0$, and $-\iu a_5>0$, we have $a_{g,0}>0$. This proves \eqref{p2:eq:hg-main} and \eqref{p2:eq:hg-leading}. If $(5g-7-k)/2$ is a negative integer, the corresponding exponent of $\zeta$ is a nonnegative integer, and the transfer coefficient is therefore zero by Theorem~\ref{p2:thm:transfer}, proving \eqref{p2:eq:hg-vanishing}. Taking $d$-th roots in the three expansions gives \eqref{p2:eq:root-limit}.
\end{proof}

\begin{corollary}\label{p2:cor:iterated-limits}
Denote $\alpha:=-\iu a_5=\abs{a_5}>0$. For $g\geq2$, define
\begin{equation}\label{p2:eq:Ag}
  A_g
  :=
  \left(\frac{25}{24}\right)^g
  \frac{2^{g-1}\sqrt{3/5}\,((g-1)!)^2
  }{
    \pi^2(5g-5)(5g-3)
    (15\alpha)^{g-1}
    \Gamma\!\left(\frac{5g-5}{2}\right)
  }.
\end{equation}
For $g\geq2$ and $d\geq1$, put
\begin{equation}\label{p2:eq:Rgd}
  \mathcal R_{g,d}
  :=
  \frac{N_{g,d}}{
    (3d-1+g)!\,
    e^{-dx_0}
    d^{(5g-7)/2}
    A_g
  }.
\end{equation}
Then
\begin{equation}\label{p2:eq:noncommuting-iterated-limits}
  \lim_{g\to\infty}\lim_{d\to\infty}\mathcal R_{g,d}=1,
  \qquad
  \lim_{d\to\infty}\lim_{g\to\infty}\mathcal R_{g,d}=0.
\end{equation}
More precisely, if $h=g-1$, then
\begin{equation}\label{p2:eq:ag0-large-genus-refined}
\begin{aligned}
  a_{g,0}
  ={}&
  \frac{\sqrt3}{24\pi^{3/2}}
  h^{-1/2}
  \left(
    \frac{\sqrt{2e/5}}{45\alpha\sqrt h}
  \right)^h
  \left(
    1-\frac4{15h}+\frac{26}{225h^2}+O(h^{-3})
  \right),
\end{aligned}
\end{equation}
as $g\to\infty$. Consequently, if
\begin{equation}\label{p2:eq:Sgd}
\begin{aligned}
  \mathcal S_{g,d}
  :={}&
  \frac{\sqrt3}{24\pi^{3/2}}
  \frac{(3d-1+g)!\,e^{-dx_0}}{d\sqrt{g-1}}
  \left(
    \frac{\sqrt{2e/5}}{45\alpha}
    \frac{d^{5/2}}{\sqrt{g-1}}
  \right)^{g-1},
\end{aligned}
\end{equation}
then
\begin{equation}\label{p2:eq:stirling-iterated-limit}
  \lim_{g\to\infty}
  \lim_{d\to\infty}
  \frac{N_{g,d}}{\mathcal S_{g,d}}
  =1.
\end{equation}
\end{corollary}

\begin{proof}
Equations \eqref{p2:eq:hg-leading} and \cite[Corollary~5.7]{Recursions_and_asymptotics_of_intersection_numbers} give
\[
  a_{g,0}
  =
  A_g
  \left(
    1-\frac{49}{3750g^3}
    -\frac{49}{1250g^4}
    +O(g^{-5})
  \right).
\]
For each fixed $g\geq2$, Theorem~\ref{p2:thm:main} gives
\[
  \lim_{d\to\infty}\mathcal R_{g,d}
  =
  \frac{a_{g,0}}{A_g},
\]
and the first identity in \eqref{p2:eq:noncommuting-iterated-limits} follows by letting $g\to\infty$.

If $C\subset\P{2}$ is an irreducible plane curve of degree $d$, then its geometric genus satisfies $g(C)\le p_a(C)=(d-1)(d-2)/2$. By the enumerative interpretation \cite[Corollary~5.2]{Intersection_theory_on_M_14_and_elliptic_Gromov-Witten_invariants}, with the geometric-genus terminology explained after \eqref{p2:eq:invariant}, we have
\[
 N_{g,d}=0
 \qquad
 \text{for }
 g>\frac{(d-1)(d-2)}{2}.
\]
Thus $\lim_{g\to\infty}\mathcal R_{g,d}=0$ for every fixed $d$, which proves the second identity in \eqref{p2:eq:noncommuting-iterated-limits}.

It remains to derive \eqref{p2:eq:ag0-large-genus-refined}. Writing $h=g-1$, Stirling's expansion gives
\[
  \frac{(h!)^2}{\Gamma(5h/2)}
  =
  \sqrt{5\pi}\,h^{3/2}
  \left(
    \frac{\sqrt e}{(5/2)^{5/2}\sqrt h}
  \right)^h
  \left(
    1+\frac2{15h}+\frac2{225h^2}+O(h^{-3})
  \right),
\]
whereas
\[
  \frac1{(5h)(5h+2)}
  =
  \frac1{25h^2}
  \left(
    1-\frac2{5h}+\frac4{25h^2}+O(h^{-3})
  \right).
\]
Substitution into \eqref{p2:eq:Ag} gives
\[
  A_g
  =
  \frac{\sqrt3}{24\pi^{3/2}}
  h^{-1/2}
  \left(
    \frac{\sqrt{2e/5}}{45\alpha\sqrt h}
  \right)^h
  \left(
    1-\frac4{15h}+\frac{26}{225h^2}+O(h^{-3})
  \right).
\]
The relative correction $a_{g,0}/A_g$ starts at order $g^{-3}$, and hence at order $h^{-3}$, so the same first two correction terms hold for $a_{g,0}$. This proves \eqref{p2:eq:ag0-large-genus-refined}.

Finally, combining the fixed-genus leading term in Theorem~\ref{p2:thm:main} with \eqref{p2:eq:ag0-large-genus-refined} proves \eqref{p2:eq:stirling-iterated-limit}.
\end{proof}

\begin{remark}
For fixed $g\geq2$, write the convergent local expansion at the singularity $x_0$ as
\[
  F_g(x_0+\zeta)
  =
  \sum_{k\geq0}
  B_{g,k}\,
  \zeta^{-(5g-5-k)/2}.
\]
Whenever $\frac{5g-5-k}{2}\notin\mathbb Z_{\leq0}$, Theorem~\ref{p2:thm:transfer} gives the exact identity
\begin{equation}\label{p2:eq:local-to-degree-ratio}
  \frac{a_{g,k}}{a_{g,0}}
  =
  \iu^k
  \frac{B_{g,k}}{B_{g,0}}
  \frac{
    \Gamma\!\left(\frac{5g-5}{2}\right)
  }{
    \Gamma\!\left(\frac{5g-5-k}{2}\right)
  }.
\end{equation}
When the denominator Gamma function has a pole, the corresponding transfer coefficient vanishes, in agreement with \eqref{p2:eq:hg-vanishing}. Identity \eqref{p2:eq:local-to-degree-ratio} is pointwise in $g$ and supplies no estimate uniform in the genus.
\end{remark}

\begin{remark}
All two-parameter limits above are iterated limits. In the first identity of \eqref{p2:eq:noncommuting-iterated-limits} and in \eqref{p2:eq:stirling-iterated-limit}, the limit $d\to\infty$ is taken first with $g$ fixed, and then $g\to\infty$. In the second identity of \eqref{p2:eq:noncommuting-iterated-limits}, the order is reversed: the limit $g\to\infty$ is taken first with $d$ fixed, and then $d\to\infty$. The estimates proved in this paper are not uniform in $g$, so no simultaneous asymptotic with $g=g(d)$ is asserted.
\end{remark}

\appendix

\section{Explicit formulas on the slice \texorpdfstring{$S$}{S}}\label{p2:app:explicit-R1-D}

We present some explicit formulas on the slice $S=\{t_0=0,\ t_1=z,\ t_2=1,\ Q=1\}$ used in the preceding proofs. Put
\[
 A:=F_0(z),\qquad B:=F_0'(z),\qquad C:=F_0''(z).
\]
The discriminant of the characteristic polynomial of $\E|_S$ is
\begin{align*}
\mathcal D={}&
 -24AC^4+16B^2C^3+12BC^4+36C^5\\
&{}-288A^2C^2+864AB^2C-1152ABC^2+648AC^3\\
&{}-432B^4+432B^3C-504B^2C^2+2700BC^3-1620C^4\\
&{}-864A^3+6480A^2B-7776A^2C-4536AB^2
   +27216ABC-40824AC^2\\
&{}-38988B^3+103032B^2C-90396BC^2+81648C^3\\
&{}-78732A^2+708588AB-708588AC\\
&{}-1594323B^2+3188646BC-1594323C^2.
\end{align*}
The polynomial $\mathcal D$ is irreducible in $\Q[A,B,C]$.

Define the following primitive polynomials:
\begin{align*}
\mathcal P_{11}:={}&
 -16ABC+42AC^2+12B^3-32B^2C+39BC^2-99C^3\\
&{}+108A^2-918AB+810AC+1944B^2-3402BC+1458C^2,
\end{align*}
\begin{align*}
\mathcal P_{12}:={}&
 12AC^3-8B^2C^2-6BC^3-18C^4\\
&{}+72A^2C-72AB^2-720ABC+1512AC^2\\
&{}+612B^3-1134B^2C+648BC^2-2430C^3\\
&{}+2916A^2-26244AB+26244AC\\
&{}+59049B^2-118098BC+59049C^2,
\end{align*}
\begin{align*}
\mathcal P_{13}:={}&
 6AC^4-4B^2C^3-3BC^4-9C^5\\
&{}+72A^2C^2-216AB^2C+342ABC^2-378AC^3\\
&{}+108B^4-144B^3C+243B^2C^2-648BC^3+729C^4\\
&{}+216A^3-1728A^2B+1296A^2C\\
&{}+2484AB^2-1296ABC-972AC^2\\
&{}+4131B^3-20169B^2C+26973BC^2-10935C^3,
\end{align*}
\begin{align*}
\mathcal P_{21}:={}&
 4AC^2-4B^2C+6BC^2-18C^3\\
&{}+24A^2-192AB+198AC+330B^2-495BC-27C^2,
\end{align*}
and
\begin{align*}
\mathcal P_{31}:={}&
 -4BC^2+12C^3+72AB-72AC\\
&{}-288B^2+396BC+108C^2+1458A-6561B+6561C.
\end{align*}
Then, on $\mathcal D\neq0$, the first coefficient of the $R$-matrix is
\begin{equation*}
 R_1|_S
 =
 \frac1{\mathcal D}
 \begin{pmatrix}
 -18\mathcal P_{11}&6\mathcal P_{12}&-4\mathcal P_{13}\\
 -18\mathcal P_{21}&36\mathcal P_{11}&6\mathcal P_{12}\\
 -3\mathcal P_{31}&-18\mathcal P_{21}&-18\mathcal P_{11}
 \end{pmatrix}.
\end{equation*}

\section*{Acknowledgements}
The authors thank Qingsheng Zhang for developing the code used in this work and for helpful discussions.

\section*{Conflicts of interest}
None.

\section*{Financial support}
The first author was partially supported by NSFC 12225101.

\bibliographystyle{amsalpha}
\bibliography{references_P2}

@article{A_mathematical_theory_of_quantum_cohomology,
  author  = {Ruan, Yongbin and Tian, Gang},
  title   = {A mathematical theory of quantum cohomology},
  journal = {Journal of Differential Geometry},
  volume  = {42},
  number  = {2},
  year    = {1995},
  pages   = {259--367},
  doi     = {10.4310/jdg/1214457234}
}

@book{Analytic_Combinatorics,
  author    = {Flajolet, Philippe and Sedgewick, Robert},
  title     = {Analytic Combinatorics},
  publisher = {Cambridge University Press},
  address   = {Cambridge},
  year      = {2009},
  doi       = {10.1017/CBO9780511801655},
  note      = {\url{https://ac.cs.princeton.edu/home/}}
}

@article{Arnold_conjecture_and_Gromov-Witten_invariant,
  author  = {Fukaya, Kenji and Ono, Kaoru},
  title   = {{Arnold} conjecture and {Gromov--Witten} invariant},
  journal = {Topology},
  volume  = {38},
  number  = {5},
  year    = {1999},
  pages   = {933--1048},
  doi     = {10.1016/S0040-9383(98)00042-1}
}

@article{Asymptotic_of_enumerative_invariants_in_CP2,
  author  = {Tian, Gang and Wei, Dongyi},
  title   = {Asymptotic of Enumerative Invariants in {${\mathbb{C}}P^2$}},
  journal = {Peking Mathematical Journal},
  volume  = {1},
  number  = {2},
  year    = {2018},
  pages   = {125--140},
  doi     = {10.1007/s42543-018-0004-4}
}

@article{Combinatorics_of_the_modular_group_II_the_Kontsevich_integrals,
  author  = {Itzykson, Claude and Zuber, Jean-Bernard},
  title   = {Combinatorics of the modular group {II}: the {Kontsevich} integrals},
  journal = {International Journal of Modern Physics A},
  volume  = {7},
  number  = {23},
  year    = {1992},
  pages   = {5661--5705},
  doi     = {10.1142/S0217751X92002581}
}

@incollection{Geometry_of_2D_topological_field_theories,
  author    = {Dubrovin, Boris},
  title     = {Geometry of {2D} topological field theories},
  editor    = {Francaviglia, Mauro and Greco, Silvio},
  booktitle = {Integrable Systems and Quantum Groups},
  series    = {Lecture Notes in Mathematics},
  volume    = {1620},
  publisher = {Springer},
  address   = {Berlin},
  year      = {1996},
  pages     = {120--348},
  doi       = {10.1007/BFb0094793}
}

@article{Gromov-Witten_classes_quantum_cohomology_and_enumerative_geometry,
  author  = {Kontsevich, Maxim and Manin, Yuri},
  title   = {{Gromov--Witten} classes, quantum cohomology, and enumerative geometry},
  journal = {Communications in Mathematical Physics},
  volume  = {164},
  number  = {3},
  year    = {1994},
  pages   = {525--562},
  doi     = {10.1007/BF02101490}
}

@article{Gromov-Witten_invariants_and_quantization_of_quadratic_hamiltonians,
  author  = {Givental, Alexander B.},
  title   = {{Gromov--Witten} invariants and quantization of quadratic {Hamiltonians}},
  journal = {Moscow Mathematical Journal},
  volume  = {1},
  number  = {4},
  year    = {2001},
  pages   = {551--568},
  doi     = {10.17323/1609-4514-2001-1-4-551-568}
}

@incollection{Gromov-Witten_invariants_of_general_symplectic_manifolds,
  author    = {Siebert, Bernd},
  title     = {Symplectic {Gromov--Witten} invariants},
  editor    = {Hulek, Klaus and Reid, Miles and Peters, Chris and Catanese, Fabrizio},
  booktitle = {New Trends in Algebraic Geometry},
  series    = {London Mathematical Society Lecture Note Series},
  volume    = {264},
  publisher = {Cambridge University Press},
  address   = {Cambridge},
  year      = {1999},
  pages     = {375--424},
  doi       = {10.1017/CBO9780511721540.016}
}

@article{Higher_genus_symplectic_invariants_and_sigma_model_coupled_with_gravity,
  author  = {Ruan, Yongbin and Tian, Gang},
  title   = {Higher genus symplectic invariants and sigma models coupled with gravity},
  journal = {Inventiones Mathematicae},
  volume  = {130},
  number  = {3},
  year    = {1997},
  pages   = {455--516},
  doi     = {10.1007/s002220050192},
  note    = {arXiv:alg-geom/9601005}
}

@article{Hodge_integrals_and_Gromov-Witten_theory,
  author  = {Faber, Carel and Pandharipande, Rahul},
  title   = {Hodge integrals and {Gromov--Witten} theory},
  journal = {Inventiones Mathematicae},
  volume  = {139},
  number  = {1},
  year    = {2000},
  pages   = {173--199},
  doi     = {10.1007/s002229900028}
}

@article{Intersection_theory_on_M_14_and_elliptic_Gromov-Witten_invariants,
  author  = {Getzler, Ezra},
  title   = {Intersection theory on {$\overline{\mathcal{M}}_{1,4}$} and elliptic {Gromov--Witten} invariants},
  journal = {Journal of the American Mathematical Society},
  volume  = {10},
  number  = {4},
  year    = {1997},
  pages   = {973--998},
  doi     = {10.1090/S0894-0347-97-00246-4}
}

@book{J-holomorphic_curves_and_quantum_cohomology,
  author    = {McDuff, Dusa and Salamon, Dietmar A.},
  title     = {{$J$}-Holomorphic Curves and Quantum Cohomology},
  series    = {University Lecture Series},
  volume    = {6},
  publisher = {American Mathematical Society},
  address   = {Providence, RI},
  year      = {1994},
  doi       = {10.1090/ulect/006}
}

@book{Mirror_symmetry,
  author    = {Hori, Kentaro and Katz, Sheldon and Klemm, Albrecht and Pandharipande, Rahul and Thomas, Richard and Vafa, Cumrun and Vakil, Ravi and Zaslow, Eric},
  title     = {Mirror symmetry},
  series    = {Clay Mathematics Monographs},
  volume    = {1},
  publisher = {American Mathematical Society and Clay Mathematics Institute},
  address   = {Providence, RI and Cambridge, MA},
  year      = {2003},
  isbn      = {978-0-8218-2955-4}
}

@book{Mirror_symmetry_and_algebraic_geometry,
  author    = {Cox, David A. and Katz, Sheldon},
  title     = {Mirror symmetry and algebraic geometry},
  series    = {Mathematical Surveys and Monographs},
  volume    = {68},
  publisher = {American Mathematical Society},
  address   = {Providence, RI},
  year      = {1999},
  doi       = {10.1090/surv/068}
}

@misc{NIST_DLMF,
  key          = {{\relax DLMF}},
  title        = {{NIST Digital Library of Mathematical Functions}},
  howpublished = {\url{https://dlmf.nist.gov/}, Release 1.2.7 of 2026-06-15},
  url          = {https://dlmf.nist.gov/},
  note         = {F.~W.~J. Olver, A.~B. Olde Daalhuis, D.~W. Lozier,
                  B.~I. Schneider, R.~F. Boisvert, C.~W. Clark,
                  B.~R. Miller, B.~V. Saunders, H.~S. Cohl,
                  and M.~A. McClain, eds.}
}

@misc{On_asymptotic_behavior_of_GW_invariants,
  author       = {Zinger, Aleksey},
  title        = {On asymptotic behavior of {GW}-invariants},
  howpublished = {Informal note},
  year         = {2013},
  note         = {Dated 23 September 2013, \url{https://www.math.stonybrook.edu/~azinger/research/GWest0.pdf}}
}

@article{On_degree_zero_elliptic_orbifold_Gromov-Witten_invariants,
  author  = {Tseng, Hsian-Hua},
  title   = {On {Degree-0} elliptic orbifold {Gromov--Witten} invariants},
  journal = {International Mathematics Research Notices},
  volume  = {2011},
  number  = {11},
  year    = {2011},
  pages   = {2444--2468},
  doi     = {10.1093/imrn/rnq158}
}

@article{Pseudo_holomorphic_curves_in_symplectic_manifolds,
  author  = {Gromov, Mikhail},
  title   = {Pseudo holomorphic curves in symplectic manifolds},
  journal = {Inventiones Mathematicae},
  volume  = {82},
  number  = {2},
  year    = {1985},
  pages   = {307--347},
  doi     = {10.1007/BF01388806}
}

@incollection{Quantum_intersection_rings,
  author    = {Di Francesco, Philippe and Itzykson, Claude},
  title     = {Quantum intersection rings},
  editor    = {Dijkgraaf, Robbert H. and Faber, Carel F. and van der Geer, Gerard B. M.},
  booktitle = {The Moduli Space of Curves},
  series    = {Progress in Mathematics},
  volume    = {129},
  publisher = {Birkh{\"a}user Boston},
  address   = {Boston, MA},
  year      = {1995},
  pages     = {81--148},
  doi       = {10.1007/978-1-4612-4264-2_4}
}

@article{Recursions_and_asymptotics_of_intersection_numbers,
  author  = {Liu, Kefeng and Mulase, Motohico and Xu, Hao},
  title   = {Recursions and asymptotics of intersection numbers},
  journal = {International Journal of Mathematics},
  volume  = {27},
  number  = {9},
  year    = {2016},
  pages   = {1650072},
  doi     = {10.1142/S0129167X16500725}
}

@article{Relations_on_Mgn_via_3-spin_structures,
  author  = {Pandharipande, Rahul and Pixton, Aaron and Zvonkine, Dimitri},
  title   = {Relations on {$\overline{\mathcal{M}}_{g,n}$} via {3}-spin structures},
  journal = {Journal of the American Mathematical Society},
  volume  = {28},
  number  = {1},
  year    = {2015},
  pages   = {279--309},
  doi     = {10.1090/S0894-0347-2014-00808-0}
}

@incollection{Some_conjectures_on_the_asymptotic_behavior_of_Gromov-Witten_invariants,
  author    = {Zinger, Aleksey},
  title     = {Some conjectures on the asymptotic behavior of {Gromov--Witten} invariants},
  editor    = {Ji, Lizhen and Wu, Baosen and Yau, Shing-Tung},
  booktitle = {Handbook for Mirror Symmetry of Calabi--Yau and Fano Manifolds},
  series    = {Advanced Lectures in Mathematics},
  volume    = {47},
  publisher = {Higher Education Press and International Press},
  address   = {Beijing and Somerville, MA},
  year      = {2019},
  pages     = {523--550},
  note      = {arXiv:1610.02971v2}
}

@article{The_intrinsic_normal_cone,
  author  = {Behrend, Kai and Fantechi, Barbara},
  title   = {The intrinsic normal cone},
  journal = {Inventiones Mathematicae},
  volume  = {128},
  number  = {1},
  year    = {1997},
  pages   = {45--88},
  doi     = {10.1007/s002220050136}
}

@article{The_singularity_of_Kontsevichs_solution_for_QHCP2,
  author  = {Guzzetti, Davide},
  title   = {The Singularity of {Kontsevich}'s Solution for {$\mathrm{QH}^*(\mathrm{CP}^2)$}},
  journal = {Mathematical Physics, Analysis and Geometry},
  volume  = {8},
  number  = {1},
  year    = {2005},
  pages   = {41--58},
  doi     = {10.1007/s11040-004-0936-z}
}

@article{The_structure_of_2D_semi-simple_field_theories,
  author  = {Teleman, Constantin},
  title   = {The structure of {2D} semi-simple field theories},
  journal = {Inventiones Mathematicae},
  volume  = {188},
  number  = {3},
  year    = {2012},
  pages   = {525--588},
  doi     = {10.1007/s00222-011-0352-5}
}

@article{Virtual_moduli_cycles_and_Gromov-Witten_invariants_of_algebraic_varieties,
  author  = {Li, Jun and Tian, Gang},
  title   = {Virtual moduli cycles and {Gromov--Witten} invariants of algebraic varieties},
  journal = {Journal of the American Mathematical Society},
  volume  = {11},
  number  = {1},
  year    = {1998},
  pages   = {119--174},
  doi     = {10.1090/S0894-0347-98-00250-1}
}

@incollection{Virtual_moduli_cycles_and_Gromov-Witten_invariants_of_general_symplectic_manifolds,
  author    = {Li, Jun and Tian, Gang},
  title     = {Virtual moduli cycles and {Gromov--Witten} invariants of general symplectic manifolds},
  editor    = {Stern, Ronald J.},
  booktitle = {Topics in Symplectic 4-Manifolds (Irvine, CA, 1996)},
  series    = {First International Press Lecture Series},
  volume    = {1},
  publisher = {International Press},
  address   = {Cambridge, MA},
  year      = {1998},
  pages     = {47--83},
  mrnumber  = {1635695},
  note      = {arXiv:alg-geom/9608032}
}

\end{document}